\documentclass{article}

\usepackage[english]{babel}

\usepackage[letterpaper,top=3cm,bottom=3cm,left=3cm,right=3cm,marginparwidth=1.75cm]{geometry}

\usepackage{amsmath}
\usepackage{graphicx}
\usepackage{cite}
\usepackage[colorlinks=true, allcolors=blue]{hyperref}
\usepackage[noblocks]{authblk}
\usepackage{algorithm}

\usepackage{algpseudocode}
\usepackage{amsthm}
\usepackage{arydshln}
\newlength\figureheight
\newlength\figurewidth
\usepackage{amssymb}
\usepackage{mathdots}
\usepackage{multirow}
\usepackage{float}

\usepackage{algorithm}
\usepackage{algpseudocode}
\usepackage{mathtools}
\usepackage{tikz}
\usepackage{tikzscale}
\usetikzlibrary{matrix}
\usepackage{pgfplots}
\usepackage{pdfpages}
\usepackage{hyperref}
\usepackage{mathtools}
\usepackage[noblocks]{authblk}
\usepackage{verbatim}
\usepackage{tikz}
\usetikzlibrary{plotmarks}
\usepackage{subcaption}

\usepackage{tablefootnote}

\newtheorem{theorem}{Theorem}[section]

\newtheorem{corollary}[theorem]{Corollary}

\newtheorem{example}{Example}[section]
\newtheorem{remark}{Remark}[section]

\title{On a spectral solver for highly oscillatory and non-smooth solutions of a class of linear fractional differential systems}
\author[1a]{Amin Faghih}
	
 \affil[1]{Department of Computer Science, KU Leuven, Leuven, Belgium}

\affil[a]{\texttt{amin.faghih@kuleuven.be}}

\begin{document}
\maketitle
\begin{abstract}
This study discusses a class of linear systems of fractional differential equations with non-constant coefficients, with a particular focus on problems exhibiting highly oscillatory and non-smooth behavior. We first establish the regularity properties of the solutions under specific conditions on the input data. A spectral Galerkin method based on M\"{u}ntz-Jacobi functions is developed that efficiently handle the non-smooth and highly oscillatory solutions. A key advantage of the proposed approach is the ability to compute the approximate solution via recurrence relations, avoiding the need to solve complex algebraic systems. Moreover, the method remains stable even at higher approximation degrees, effectively capturing highly oscillatory solutions with high accuracy. The well-known exponential accuracy is established in the $L^2$-norm, and some numerical examples are provided to demonstrate both the validity of the theoretical analysis and the efficiency of the proposed algorithm.
\end{abstract}
	\renewcommand{\thefootnote}{\fnsymbol{footnote}}

\section{Introduction}
Fractional calculus (FC) has increasingly become an essential framework for modeling diverse phenomena in science and engineering. In particular, fractional differential equations (FDEs) have proven invaluable for describing processes that exhibit memory-dependent behavior. Unlike the integer-order differential equations, FDEs offer a great opportunity for modeling and simulating multi-physics phenomena. During the last few decades there has been an upsurge of intense research exploring various aspects of the theory and applications of FC and FDEs.   

Kilbas et al. \cite{Kilbas2006} present a detailed introduction to the theory and applications of FDEs. Diethelm \cite{39} offers a clear introduction to FC before focusing on Caputo-type FDEs. This book contains mathematically sound theory and relevant applications. Influential works in applications include Oustaloup \cite{Oustaloup1995} in control theory, Hilfer \cite{Hilfer2000} and Hermann \cite{Herrmann2011} in physics. Podlubny \cite{Podlubny1999} has become a standard reference, covering applications in mechanics, physics, and engineering. Magin \cite{Magin2006} explores bioengineering applications, while Atanackovic et al. \cite{Atanackovic2014} present applications in mechanics. Baleanu et al. \cite{Baleanu2012} present fractional models and numerical methods for solving FDEs. Tarasov \cite{Tarasov2010} discusses applications in statistical physics, condensed matter, and quantum dynamics. Other relevant contributions can be found in the books by Ortigueira \cite{Ortigueira2011}, Baleanu et al. \cite{Baleanu2012}, Petráš \cite{Petras2011}, and Sabatier et al. \cite{Sabatier2007}. Mainardi \cite{Mainardi2010} is a standard reference for the application of FC in viscoelasticity and wave motion, and Uchaikin \cite{Uchaikin2013} provides detailed motivation for FDEs in various branches of physics. Yang et al. \cite{Yang2016}, and Gholami Bahador et al. \cite{MR4512456,Bahador2023} successfully use fractional-order methods in image processing and image denoising. 

There are diverse physical phenomena whose individual components interact closely. Consequently, the mathematical modeling of such phenomena may lead to systems of fractional differential equations (SFDEs). Due to the non-local nature of fractional derivatives,
the computations involved in solving SFDEs are tedious and time consuming. Developing numerical methods for solving SFDEs has been a subject of intense research at present. Table \ref{table:methods} provides references for the different methods for solving SFDEs. 
\begin{table}[H]
\centering
\setlength{\tabcolsep}{9pt}
\begin{tabular}{l|p{6.7cm}|c|c}
\textbf{Type} & \textbf{Method} & \textbf{Linear} & \textbf{Non-linear} \\ \hline
\rule{0pt}{10pt}
\multirow{8}{*}{Single-order SFDEs} 
  & fractional Laguerre pseudo-spectral \cite{MR3254684} & \checkmark & \checkmark \\
  & Petrov-Galerkin \cite{MR4292976} & \checkmark &  \\
  & fractional-order Jacobi Tau \cite{Bhrawy2016} & \checkmark & \checkmark \\
  & hybrid non-polynomial collocation \cite{Ferras2020} & \checkmark &  \\
  & spline collocation \cite{MR4097975} & \checkmark &  \\
  & Jacobi-Gauss collocation \cite{MR4085162} & \checkmark & \checkmark \\
  & central part interpolation \cite{Lillemae2025} & \checkmark &  \\
 & fractional Hamiltonian boundary value \cite{BrugnanoGurioliIavernaro2025,BrugnanoGurioliIavernaroVikerpuur2025,BrugnanoGurioliIavernaroVikerpuur2025FDE} & \checkmark & \checkmark \\[12pt]

\multirow{7}{*}{Multi-order SFDEs\tablefootnote{Multi-order SFDEs refer to systems of fractional differential equations where each equation involves a fractional derivative of a distinct order.}}
  & Legendre wavelet \cite{Chen2015} & \checkmark & \checkmark \\
  & Bernoulli wavelet \cite{Wang2019} & \checkmark & \checkmark \\
  & Chebyshev collocation \cite{Khader2013,Khader2015} & \checkmark & \checkmark \\
  & fractional Jacobi collocation \cite{36} & \checkmark & \\
  & first kind Chebyshev Tau \cite{Atabakzadeh2013} & \checkmark & \checkmark \\
  & second kind Chebyshev Tau \cite{35} & \checkmark &  \\
  & fractional Galerkin \cite{MR4398496} &  & \checkmark \\
\end{tabular}
\caption{Summary of numerical procedures for solving SFDEs}
\label{table:methods}
\end{table}
Many existing approaches face notable drawbacks, listed below, that limit their efficiency and reliability:
\begin{itemize}
    \item \textbf{Mismatch between basis functions and solution behavior:} The singular nature of fractional derivative operators is often not adequately addressed when selecting basis functions. As a result, many methods rely on infinitely smooth functions (such as classical polynomials) that fail to capture the solution’s actual asymptotic and regularity characteristics. Achieving truly high-order accuracy requires aligning the smoothness of the basis with the inherent behavior of the exact solution, a condition often disregarded in current literature \cite{MR4097975,Atabakzadeh2013,Khader2013,Khader2015,Chen2015,Wang2019}.

    \item \textbf{High computational complexity:} In much of the literature on nonlinear SFDEs and on linear SFDEs with non-constant coefficients, collocation techniques are widely used: the residual is enforced to vanish at a set of collocation points \cite{MR4085162,36}. This strategy typically leads to a complex nonlinear algebraic system, in the linear case, to a dense linear system. Both scenarios can impose large computational costs and may degrade the accuracy of the numerical solution. These drawbacks become more pronounced tackling highly oscillatory problems, such as certain fractional-order mechanics models, real-world problems, or problems on long domains, where satisfactory accuracy is typically achieved only at very high approximation degrees.
\end{itemize}
The goal of this research is to address these challenges by implementing a Müntz-Jacobi Galerkin method for solving the following class of linear SFDEs with non-constant coefficients
\begin{equation}\label{eq1}
\begin{cases}
D_C^{\theta_j} v_j(t)=\sum\limits_{r=1}^{n}{p_{j,r}(t)~ v_r(t)}+p_{j,n+1}(t),\quad j=1,2,\ldots,n,\\
v_j^{(k)}(0)=v_{j,0}^{(k)},~k=0,1,\ldots, \lceil \theta_j \rceil-1,~ t\in \Lambda=[0,T],
\end{cases}
\end{equation}
where $\theta_{j} = \frac{\gamma_{j}}{q_{j}} \in \mathbb{Q}^{+}$ with co-prime integers $\gamma_{j} \ge 1$ and $q_{j} \ge 2$, $T \in \mathbb{R}^{+}$ (finite), and $\lceil \cdot \rceil$ stands for the ceiling function\footnote{$\lceil \theta_{j} \rceil$ denotes the smallest integer greater than or equal to $\theta_{j}$.}. The $p_{j,r}(t)$ are continuous functions on $\Lambda$, and $v_{j}(t)$ are the unknowns. $D^{\theta_{j}}_{C}$ is the Caputo fractional derivative of order $\theta_{j}$, defined by
\begin{equation*}
D^{\theta_j}_C (.)=I^{\lceil \theta_j \rceil-\theta_{j}}\partial_x^{\lceil \theta_j \rceil}(.),
\end{equation*}
in which $I^{\lceil \theta_j \rceil-\theta_{j}}$ denotes the Riemann-Liouville fractional integral operator of order $\lceil \theta_j \rceil-\theta_{j}$
 \cite{Kilbas2006,39,Podlubny1999}.
 
This study presents an in-depth analysis of the regularity properties of equation \eqref{eq1}, focusing particularly on cases where the problem exhibits highly oscillatory behavior. To tackle \eqref{eq1}, an efficient spectral method based on M\"{u}ntz-Jacobi functions, which were introduced by Shen and Wang \cite{r35}, is employed, as they reflect the asymptotic behavior of the solutions of \eqref{eq1}. Spectral techniques are well known for providing highly accurate approximations when applied to smooth problems. Nevertheless, they have certain limitations, such as the challenge of solving potentially ill-conditioned or ambiguous algebraic systems, and the significant loss of accuracy when dealing with problems that have non-smooth solutions. The approach adopted here overcomes these challenges by maintaining spectral accuracy for non-smooth and oscillatory solutions and allowing the approximate solution to be computed efficiently via recurrence relations, all without the necessity of solving complex algebraic systems.

The remainder of the paper is arranged as follows. Section \ref{secA} focuses on the regularity analysis of the solution to \eqref{eq1}. In Section \ref{sec1}, we introduce the M\"{u}ntz-Jacobi method for solving \eqref{eq1} under the assumptions of the regularity theorem. The unique solvability and computational complexity of the numerical solution are also addressed. In particular, Section \ref{sec4} presents a detailed investigation of the error estimates. Finally, the numerical experiments in Section \ref{sec5} demonstrate the effectiveness of our proposed algorithm and include a performance comparison with methods from the literature.
 \section{Regularity analysis}\label{secA}
For the existence and uniqueness of the solution to \eqref{eq1}, the continuity of the functions $p_{j,r}(t)$ is sufficient, and the solution inherits this continuity. This follows directly from Theorem 8.11 in \cite{39}. The following theorem provides a series representation of the solution of \eqref{eq1} around the origin and specifies the conditions under which the solution exhibits highly oscillatory behavior.
 \begin{theorem}[Regularity]\label{th202} Let $p_{j,r} : \mathbb{R} \to \mathbb{C}$ be periodic, highly oscillatory continuous functions of the form
\[
p_{j,r}(t) = f_{j,r}(t) \, e^{i \omega t}, \quad \omega \gg 1, \quad r=1,2,\ldots,n+1,
\]
which can be written as
\[
p_{j,r}(t) = \bar{p}_{j,r}\!\left(t^{1/q}\right),
\]
where $\bar{p}_{j,r}$ are analytic in a neighborhood of the origin. Then the series representation of the solution $v_j(t)$ of the equation \eqref{eq1}, in a neighborhood of the origin, is given by
\begin{equation}\label{eq1500}
v_j(t)=\psi_{j}(t)+\sum\limits_{\mu=\theta_{j}q}^{\infty}{\bar{v}_{j,\mu}~t^{\frac{\mu}{q}}}, \quad j=1,2,\ldots,n,
\end{equation}
in which $\psi_{j}(t)=\sum\limits_{k=0}^{\lceil \theta_j \rceil -1}{\frac{v_{j,0}^{(k)}}{k!}t^k}$, $q$ denotes the least common multiple of the denominators $q_{j}$, and the $\bar{v}_{j,\mu}$ are known coefficients.
\end{theorem}
\begin{proof}
First we consider the following representation of $v_j(t)$
\begin{equation}\label{eee3}
v_j(t)=\sum\limits_{\mu=0}^{\infty}{\bar{v}_{j,\mu}~t^{\frac{\mu}{q}}},\quad j=1,2,\ldots,n.
\end{equation}
The unknown coefficients $\{\bar{v}_{j,\mu}\}_{j=1}^{n}$ are chosen such that the representation \eqref{eq1500} converges and satisfies \eqref{eq1}. This is accomplished by employing the series expansion of $\bar{p}_{j,r}(t^{1/q})$, $j=1,2,\ldots,n$ around the origin, viz.
\begin{eqnarray}\label{e54-11}
p_{j,r}(t)&=&\bar{p}_{j,r}(t^{1/q})=\sum\limits_{\mu_{1},\mu_{2}=0}^{\infty}{f_{j,r,\mu_{1}}\frac{(i \omega)^{\mu_2}}{\mu_2!}~t^{\frac{\mu_{1}+\mu_{2}}{q}}},\quad r=1,2,...,n+1.
\end{eqnarray}
Considering uniform convergence and substituting \eqref{eee3} and \eqref{e54-11} into \eqref{eq1}, the coefficients 
\(\{\bar{v}_{j,\mu}\}_{j=1}^{n}\) satisfy the following equality
\begin{eqnarray*}
\nonumber\sum\limits_{\mu=0}^{\infty}{\bar{v}_{j,\mu}~t^{\frac{\mu}{q}}}=\psi_{j}(t)&+&\sum\limits_{r=1}^{n}~\sum\limits_{\mu_{1},\mu_{2}=0}^{\infty}~\sum\limits_{\mu=0}^{\infty}{f_{j,r,\mu_{1}}~\bar{v}_{r,\mu}~\frac{(i \omega)^{\mu_2}}{\mu_2!}~I^{\theta_{j}}\bigg(t^{\frac{\mu_{1}+\mu_{2}+\mu}
{q}}\bigg)}\\
&+&\sum\limits_{\mu_{1},\mu=0}^{\infty}{f_{j,n+1,\mu_{1}}~\frac{(i \omega)^{\mu}}{\mu!}~I^{\theta_{j}}\bigg(t^{\frac{\mu_{1}+\mu}
{q}}\bigg)},
\end{eqnarray*}
which can be rewritten as 
\begin{eqnarray}\label{eq148}
\nonumber\sum\limits_{\mu=0}^{\infty}{\bar{v}_{j,\mu}~t^{\frac{\mu}{q}}}=\psi_{j}(t)&+&\sum\limits_{r=1}^{n}~\sum\limits_{\mu_{1},\mu_{2}=0}^{\infty}~\sum\limits_{\mu=0}^{\infty}{f_{j,r,\mu_{1}}~\bar{v}_{r,\mu}~\xi_{1,j}~\frac{(i \omega)^{\mu_2}}{\mu_2!}~t^{\frac{\mu_{1}+\mu_{2}+\mu}
{q}+\theta_{j}}}\\
&+&\sum\limits_{\mu_{1},\mu=0}^{\infty}{f_{j,n+1,\mu_{1}}~\xi_{2,j}~\frac{(i \omega)^{\mu}}{\mu!}~t^{\frac{\mu_{1}+\mu}
{q}+\theta_{j}}},
\end{eqnarray}
in which 
\[\xi_{1,j}=\frac{\Gamma(\frac{\mu_{1}+\mu_{2}+\mu}{q}+1)}
{\Gamma(\frac{\mu_{1}+\mu_{2}+\mu}{q}+\theta_{j}+1)}, \quad \quad \xi_{2,j}=\frac{\Gamma(\frac{\mu_{1}+\mu}{q}+1)}
{\Gamma(\frac{\mu_{1}+\mu}{q}+\theta_{j}+1)}.\]
Substituting \(\mu = \mu - \mu_1 - \mu_2 - \theta_j q\) in the first series, and \(\mu = \mu - \mu_1 - \theta_j q\) in the second series on the right-hand side of \eqref{eq148}, we obtain
\begin{eqnarray}\label{eq149}
\nonumber\sum\limits_{\mu=0}^{\infty}{\bar{v}_{j,\mu}~t^{\frac{\mu}{q}}}=\psi_{j}(t)&+&\sum\limits_{r=1}^{n}~\sum\limits_{\mu_{1},\mu_{2}=0}^{\infty}~\sum\limits_{\mu=\mu_1+\mu_2+\theta_{j}q}^{\infty}{f_{j,r,\mu_{1}}~\bar{v}_{r,\mu-\mu_1-\mu_2-\theta_{j}q}~\bar{\xi}_{1,j}~\frac{(i \omega)^{\mu_2}}{\mu_2!}~t^{\frac{\mu}
{q}}}\\
&+&\sum\limits_{\substack{\mu_{1}=0 \\ \mu=\mu_{1}+\theta_{j}q}}^{\infty}{f_{j,n+1,\mu_{1}}~\bar{\xi}_{2,j}~\frac{(i \omega)^{\mu-\mu_{1}-\theta_{j}q}}{(\mu-\mu_{1}-\theta_{j}q)!}~t^{\frac{\mu}
{q}}},
\end{eqnarray}
where
\[\bar{\xi}_{1,j}=\frac{\Gamma(\frac{\mu}{q}-\theta_{j}+1)}
{\Gamma(\frac{\mu}{q}+1)}, \quad \quad \bar{\xi}_{2,j}=\frac{\Gamma(\frac{\mu}{q}-\theta_{j}+1)}
{\Gamma(\frac{\mu}{q}+1)}.
\]
In this step, we compare the coefficients of \(t^{\frac{\mu}{q}}\) on both sides of \eqref{eq149} to determine the unknown coefficients \(\bar{v}_{j,\mu}\). For \(\mu < \theta_j q\), it is straightforward to observe that
\begin{equation*}
\bar{v}_{j,\mu}=\left\{\begin{array}{l}
\dfrac{v_{j,0}^{(\frac{\mu}{q})}}{(\frac{\mu}{q})!},\quad \quad \mu=0,q, ..., (\lceil \theta_j \rceil-1)q,\\
0,\hspace{2 cm}\text{else},
\end{array}\right.
\end{equation*}
and for $\mu \geq \theta_{j} q$, we find
\begin{align}\label{coefficients}
\bar{v}_{j,\mu}=\sum\limits_{r=1}^{n}~\sum\limits_{\mu_{1},\mu_{2}=0}^{\infty}{f_{j,r,\mu_{1}}~\bar{v}_{r,\mu-\mu_1-\mu_2-\theta_{j}q}~\bar{\xi}_{1,j}~\frac{(i \omega)^{\mu_2}}{\mu_2!}}+\sum\limits_{\mu_{1}=0}^{\infty}{f_{j,n+1,\mu_{1}}~\bar{\xi}_{2,j}~\frac{(i \omega)^{\mu-\mu_{1}-\theta_{j}q}}{(\mu-\mu_{1}-\theta_{j}q)!}}.
\end{align}
We see that this is a recurrence relation: It states that its left-hand side, viz. $\bar{v}_{j,\mu}$, only depends on coefficients $\bar{v}_{j,\mu-\mu_1-\mu_2-\theta_{j}q}$ where
\[
0 \leq \mu-\mu_1-\mu_2-\theta_{j}q<\mu.
\]
This shows that the series representation \eqref{eq1500} provides a unique solution to the main problem \eqref{eq1}. Moreover, its uniform and absolute convergence in a neighborhood of the origin is a consequence of Theorem 3.2 in \cite{36}.
\end{proof}
Theorem \ref{th202} leads to the following series expansion of $v_{j}(t)$.
\begin{corollary}\label{Cons}
Under the assumptions of Theorem \ref{th202}, we have
    \begin{equation}\label{seriesw}
      v_j(t)= \psi_{j}(t)+\sum\limits_{\mu=\theta_{j}q}^{\infty}{\bold{G}_{\mu-\theta_{j}q}(i\omega)~t^{\frac{\mu}{q}}}, \quad j=1,2,\ldots,n,
    \end{equation}
    where $\bold{G}_{j,\mu-\theta_{j}q}(i\omega)=\sum\limits_{k=0}^{\mu-\theta_{j}q}a_{j,k}\frac{(iw)^k}{k!}$, and the coefficients $a_{j,k}$ depend on the data \(\{f_{j,r,\mu_1}\}\), the parameters \(\{\bar\xi_{\ell,j}\}_{\ell=1}^{2}\), and the known coefficients \(\{\bar v_{r,\nu}\}\), $\nu < \mu$.
\end{corollary}
\begin{proof}
This follows directly from Theorem~\ref{th202} upon substituting \eqref{coefficients} into \eqref{eq1500}.
\end{proof}
To clarify the structure of the series expansions \eqref{eq1500} and \eqref{seriesw}, we give the following example.
\begin{example}
    Consider the system \eqref{eq1} with fractional orders $\theta_{1}=\frac{1}{2}$ and $\theta_{2}=\frac{3}{2}$, together with the initial data $v_{1,0}^{(0)}$, $v_{2,0}^{(0)}$, and $v_{2,0}^{(1)}$. The given force functions are assumed to be continuous and highly oscillatory of the form
    \[
p_{j,r}(t) = f_{j,r}(t) \, e^{i \omega t}, \quad \omega \gg 1, \quad \quad j=1,2, \quad r=1,2,3.
\]
In this setting $q=2$ and the solutions $v_{1}(t)$ and $v_{2}(t)$ admit the following series representations
\begin{equation}\label{Ex0}
v_1(t)=\psi_{1}(t)+\sum\limits_{\mu=1}^{\infty}{\bar{v}_{1,\mu}~t^{\frac{\mu}{2}}},\quad v_2(t)=\psi_{2}(t)+\sum\limits_{\mu=3}^{\infty}{\bar{v}_{2,\mu}~t^{\frac{\mu}{2}}},
\end{equation}
where $\psi_{1}(t)=v_{1,0}^{(0)}$ and $\psi_{2}(t)=v_{2,0}^{(0)}+v_{2,0}^{(1)}t$. Following \eqref{coefficients}, the coefficients $\bar{v}_{1,\mu}$ for $\mu \geq 1$, and $\bar{v}_{2,\mu}$ for $\mu \geq 3$, satisfy recursive relations of the form
\begin{eqnarray*}
\bar{v}_{1,\mu}&=&\sum\limits_{r=1}^{2}~\sum\limits_{\mu_{1},\mu_{2}=0}^{\infty}{f_{1,r,\mu_{1}}~\bar{v}_{r,\mu-\mu_1-\mu_2-1}~\bar{\xi}_{1,1}~\frac{(i \omega)^{\mu_2}}{\mu_2!}}+\sum\limits_{\mu_{1}=0}^{\infty}{f_{1,3,\mu_{1}}~\bar{\xi}_{2,1}~\frac{(i \omega)^{\mu-\mu_{1}-1}}{(\mu-\mu_{1}-1)!}},\\
\bar{v}_{2,\mu}&=&\sum\limits_{r=1}^{2}~\sum\limits_{\mu_{1},\mu_{2}=0}^{\infty}{f_{2,r,\mu_{1}}~\bar{v}_{r,\mu-\mu_1-\mu_2-3}~\bar{\xi}_{1,2}~\frac{(i \omega)^{\mu_2}}{\mu_2!}}+\sum\limits_{\mu_{1}=0}^{\infty}{f_{2,3,\mu_{1}}~\bar{\xi}_{2,2}~\frac{(i \omega)^{\mu-\mu_{1}-3}}{(\mu-\mu_{1}-3)!}}.
\end{eqnarray*}
Substituting $\bar{v}_{1,\mu}$ and $\bar{v}_{2,\mu}$ into \eqref{Ex0} yields the representation \eqref{seriesw}.
\begin{eqnarray*}
     v_1(t)&=& \psi_{1}(t)+\bold{G}_{1,0}(i\omega)~t^{\frac{1}{2}}+\bold{G}_{1,1}(i\omega)~t+{\bold{G}_{1,2}(i\omega)~t^{\frac{3}{2}}}+\ldots,\\
     v_2(t)&=& \psi_{2}(t)+\bold{G}_{2,0}(i\omega)~t^{\frac{3}{2}}+\bold{G}_{2,1}(i\omega)~t^{2}+{\bold{G}_{2,2}(i\omega)~t^{\frac{5}{2}}}+\ldots,
\end{eqnarray*}
where 
\[
\mathbf{G}_{1,\mu-1}(i\omega)=\sum_{k=0}^{\mu-1} a_{1,k}\,\frac{(i\omega)^{k}}{k!},
\qquad
\mathbf{G}_{2,\mu-3}(i\omega)=\sum_{k=0}^{\mu-3} a_{2,k}\,\frac{(i\omega)^{k}}{k!}.
\]
\end{example}
From Corollary \ref{Cons} it follows that the solutions of \eqref{eq1} not only inherit the reduced regularity of the functions $p_{j,r}(t)$ at the origin (the derivative of order $\lceil \theta_{j} \rceil$ of $v_{j}(t)$ may be discontinuous there), but also their highly oscillatory behavior. This combination of reduced regularity and oscillations makes the design of high-order spectral methods particularly challenging. Achieving an accurate approximation requires a careful choice of basis functions that reproduce the asymptotic behavior of the exact solutions, since standard infinitely smooth bases fail in this setting. In addition, stability considerations are crucial, because the method can attain high accuracy only for sufficiently large approximation degrees, where the oscillatory nature of the solution becomes dominant. In the following section, we construct our numerical approach under the assumptions of Theorem \ref{th202} in order to address these difficulties.
\section{Numerical method}\label{sec1}
This section presents the M\"untz-Jacobi Galerkin method. We begin with an introduction to M\"untz-Jacobi functions \cite{r35}, which play a key role in the subsequent development.
\subsection{M\"{u}ntz-Jacobi functions}
First, we recall the shifted Jacobi polynomials $J_i^{(\alpha,\beta)}(s)$ for $\alpha, \beta > -1$, which form an orthogonal system on $I=[0,1]$ with respect to the weight function $w^{(\alpha,\beta)}(s) = s^{\beta}(1-s)^{\alpha}$ \cite{29}
\begin{equation*}
\big(J_i^{(\alpha,\beta)}, J_j^{(\alpha,\beta)}\big)_{w^{(\alpha,\beta)}} = \int_{I} J_i^{(\alpha,\beta)}(s) J_j^{(\alpha,\beta)}(s) w^{(\alpha,\beta)}(s)\, ds = \zeta_i^{(\alpha,\beta)} \delta_{ij}, \quad i,j \ge 0,
\end{equation*}
where
\begin{equation*}
\zeta_i^{(\alpha,\beta)} = \|J_i^{(\alpha,\beta)}\|_{w^{(\alpha,\beta)}}^2 = \frac{\Gamma(i+\alpha+1)\Gamma(i+\beta+1)}{(2i+\alpha+\beta+1)\, i! \, \Gamma(i+\alpha+\beta+1)}.
\end{equation*}

Here, $\delta_{ij}$ denotes the Kronecker delta, while $(\cdot,\cdot)_{w^{(\alpha,\beta)}}$ and $\|\cdot\|_{w^{(\alpha,\beta)}}$ represent the weighted $L^2$ inner product and norm, respectively. For simplicity, we write $(\cdot,\cdot)$ and $\|\cdot\|$ when $\alpha = \beta = 0$. These polynomials also admit the following explicit representation
\begin{equation*}
J_{i}^{(\alpha,\beta)}(s)=\sum_{j=0}^{i}\Upsilon_{j}^{(\alpha,\beta,i)}s^{j},
\end{equation*}
where
\begin{equation*}
 \Upsilon_{j}^{(\alpha,\beta,i)}=\dfrac{(-1)^{i-j}\Gamma(i+\beta+1)\Gamma(i+\alpha+\beta+j+1)}{\Gamma(\beta+j+1)j!\Gamma(i+\alpha+\beta+1)(i-j)!}.
\end{equation*}

\begin{itemize}
  \item Considering $\Lambda=\{\lambda_{\ell}=\ell \eta\}_{\ell=0}^{\infty}$, $\eta \in (0,1)$, we define the M\"{u}ntz space associated with $\Lambda$ as
\begin{equation*}
  \bold{M}_{N}(\Lambda)=\operatorname{span}\{t^{\lambda_{0}},t^{\lambda_{1}},\ldots,t^{\lambda_{N}}\},\quad \quad \bold{M}(\Lambda)=\bigcup_{N=0}^{\infty}\bold{M}_{N}(\Lambda).
\end{equation*}
We introduce the one-to-one mapping $I \to I$ by
\begin{equation*}
u=u(s)=s^{\frac{1}{\eta}},\quad \quad s=s(u)=u^{\eta}.
\end{equation*}
The M\"{u}ntz-Jacobi function with index $(0,\frac{1}{\eta}-1)$ is defined by
\begin{equation}\label{eq1200}
\mathcal{J}_{i}^{(0,\frac{1}{\eta}-1)}(u)= J_{i}^{(0,\frac{1}{\eta}-1)}(s(u)),\quad i=0,1,\ldots,
\end{equation}
where $J_{i}^{(0,\frac{1}{\eta}-1)}(.)$ is the shifted Jacobi polynomial with index $(0,\frac{1}{\eta}-1)$.\\
Using the explicit formula of shifted Jacobi polynomials, we obtain
\begin{equation*}
\mathcal{J}_{i}^{(0,\frac{1}{\eta}-1)}(u)=\sum_{j=0}^{i}\Upsilon_{j}^{(0,\frac{1}{\eta}-1,i)}u^{j\eta} \in  \bold{M}_{i}(\Lambda),
\end{equation*}
where
\begin{equation*}
 \Upsilon_{j}^{(0,\frac{1}{\eta}-1,i)}=\dfrac{(-1)^{i-j}\Gamma(i+\frac{1}{\eta})\Gamma(i+\frac{1}{\eta}+j)}{\Gamma(\frac{1}{\eta}+j)j!\Gamma(i+\frac{1}{\eta})(i-j)!}.
\end{equation*}
  \item The M\"{u}ntz-Jacobi functions $\{\mathcal{J}_{i}^{(0,\frac{1}{\eta}-1)}(u)\}$ are mutually $L^{2}(I)$-orthogonal, i.e.,
\begin{equation*}
\int_{I}\mathcal{J}_{i}^{(0,\frac{1}{\eta}-1)}(u)\mathcal{J}_{j}^{(0,\frac{1}{\eta}-1)}(u)du=\hat{\zeta}_{i}^{(0,\frac{1}{\eta}-1)}\delta_{ij}, \quad  i,j\geq 0,
\end{equation*}
where $\hat{\zeta}_{i}^{(0,\frac{1}{\eta}-1)}=\frac{1}{\eta}\zeta_{i}^{(0,\frac{1}{\eta}-1)}$. These functions form a complete orthogonal system in $L^{2}(I)$. That is, for each $v(u) \in L^{2}(I)$, we have the following unique expansion
  \begin{equation*}
  v(u)=
\sum_{i=0}^{\infty}c_{i}\mathcal{J}_{i}^{(0,\frac{1}{\eta}-1)}(u),\quad c_{i}=\frac{(v,\mathcal{J}_{i}^{(0,\frac{1}{\eta}-1)})}{\|\mathcal{J}_{i}^{(0,\frac{1}{\eta}-1)}\|^{2}}.
  \end{equation*}
Therefor, we have
\begin{equation*}
 \bold{M}_{N}(\Lambda)=\textit{span}\{\mathcal{J}_{0}^{(0,\frac{1}{\eta}-1)}(u),\mathcal{J}_{1}^{(0,\frac{1}{\eta}-1)}(u),\ldots,\mathcal{J}_{N}^{(0,\frac{1}{\eta}-1)}(u)\}.
\end{equation*}
\item The M\"{u}ntz-Jacobi orthogonal projection $\pi_{N}:L^2(I)\rightarrow \bold{M}_{N}(\Lambda)$ is stated by
\[
\pi_{N}v=\sum\limits_{i=0}^{N}{c_i \mathcal{J}_{i}^{(0,\frac{1}{\eta}-1)}(u)},
\]
which has the following property
\begin{equation*}
\left(v-\pi_{N}v,\phi\right)=0, \quad \quad \forall \phi \in  \bold{M}_{N}(\Lambda).
\end{equation*}
\end{itemize}
We now derive an estimate for the truncation error $\pi_{N}v - v$. For this purpose, assume that the functions $v(u)$ and $V(s)$ are related by $v(u)=V(s(u))$. Thus, their derivatives are connected as follows
\begin{eqnarray}\label{Newd}
\nonumber D_u v:=\partial_s V(s)&=&\partial_s u~ \partial_u v,\\
\nonumber D_u^2 v:=\partial_s^2 V(s)&=&\partial_s u~\partial_u(D_u v)~,\\
\nonumber \vdots
\\
D_u^k v:=\partial_s^k V(s)&=&\partial_s u~\partial_u(\partial_s u~ \partial_u(\cdots(\partial_s u~ \partial_u v)\cdots)),
\end{eqnarray}
in which $\partial_s u=\frac{1}{\eta}(u)^{1-\eta}$.

Now, we define the space
 \[H_{\hat{w}^{(0,0)}}^k(I)=\{v:~ \|v\|_{k,\hat{w}^{(0,0)}} < \infty\},
    \]
with the associated norm and semi-norm
\begin{eqnarray*}
\|v\|_{k,\hat{w}^{(0,0)}}^2&=\sum\limits_{\ell=0}^{k}{\|D_u^\ell v\|_{\hat{w}^{(\ell,\ell)}}^2},\quad  |v|_{k,\hat{w}^{(0,0)}}&=\|D_u^k v\|_{\hat{w}^{(k,k)}},
\end{eqnarray*}
in which $\hat{w}^{(\alpha,\beta)}(u)=w^{(\alpha,\beta)}(s(u))$\footnote{Note that $\hat{w}^{(0,0)}(u) = w^{(0,0)}(s(u)) = 1$, and therefore $\|\cdot\|_{\hat{w}^{(0,0)}}$ is simply the $L^{2}$-norm with the constant weight 1.}.

Next, we state a theorem that provides an upper bound for the truncation error $\pi_{N}v - v$.
\begin{theorem}\label{thm2}
    \cite[Theorem 3.3]{r35} It follows that, for some $C>0$ independent of $N$,
\begin{equation*}
 \|\pi_{N}v-v\|_{\hat{w}^{(0,0)}} \leq C N^{-k}|v|_{k,\hat{w}^{(0,0)}},
\end{equation*}
where $v \in H_{\hat{w}^{(0,0)}}^{k}(I)$ and $k \in \mathbb{N}$.
\end{theorem}
Further properties of the M\"{u}ntz-Jacobi functions can be found in the work of Shen and Wang \cite{r35}.
\subsection{M\"{u}ntz-Jacobi Galerkin algorithm}
We proceed to implement the numerical scheme and examine its computational complexity. 

By applying $I^{\theta_j}$ to both sides of \eqref{eq1} and using $I^{\theta_j} D_C^{\theta_j} v_j(t) = v_j(t)-\psi_{j}(t)$, the equation \eqref{eq1} can be rewritten as the following system of Volterra integral equations
\begin{equation}\label{e31}
v_j(t)=\psi_{j}(t)+\sum\limits_{r=1}^{n}{I^{\theta_j}\big(p_{j,r}(t)v_r(t)\big)}+I^{\theta_j}\big(p_{j,n+1}(t)\big),\quad j =1,2,\ldots,n.
\end{equation}
Under the coordinate transformation $t = Tu$, the equation \eqref{e31} is transformed into the following system of equations on $I$
\begin{equation}\label{e367}
\bar{v}_j(u)=\bar{\psi}_{j}(u)+T^{\theta_j}\sum\limits_{r=1}^{n}{I^{\theta_j}\big(\bar{p}_{j,r}(u)\bar{v}_r(u)\big)}+T^{\theta_j}I^{\theta_j}\big(\bar{p}_{j,n+1}(u)\big),\quad j =1,2,\ldots,n,
\end{equation}
where $\bar{v}_j(u)=v_{j}(Tu)$, $\bar{\psi}_{j}(u)=\psi_{j}(Tu)$, and $\bar{p}_{j,r}(u)=p_{j,r}(Tu)$, $j=1,2,\ldots,n$, $r=1,2,\ldots,n+1$.
By setting $\eta = \frac{1}{q}$ in \eqref{eq1200}, the M\"untz-Jacobi functions $\{\mathcal{J}_i^{(0,q-1)}\}_{i \ge 0}$ are obtained as follows
\begin{equation}\label{e32}
\mathcal{J}_{i}^{(0,q-1)}(u)= J_{i}^{(0,q-1)}(s(u))=\operatorname{span}\{1,u^{\eta},\ldots,u^{i\eta}\}.
\end{equation}

Now, we present a highly accurate Galerkin solution $\bar{v}_{j,N}(u)$ to the transformed equation \eqref{e367} where the M\"{u}ntz-Jacobi functions $\{\mathcal{J}_i^{(0,q-1)}\}_{i \geq 0}$ introduced in \eqref{e32} are employed as basis functions. In this regard, we consider the Galerkin solution of \eqref{e367} as
\begin{equation}\label{eq11}
\bar{v}_{j,N}(u)=\sum\limits_{i=0}^{\infty}{c_{j,i}\mathcal{J}_i^{(0,q-1)}(u)}=\underline{c}_j\underline{\mathcal{J}}^{(0,q-1)}=
\underline{c}_j \mathcal{J} \underline{U}_u,\quad j=1,2,\dots,n,
\end{equation}
where $\underline{\mathcal{J}}^{(0,q-1)}=[\mathcal{J}_0^{(0,q-1)}(u), \mathcal{J}_1^{(0,q-1)}(u), ..., \mathcal{J}_N^{(0,q-1)}(u), ...]^{T} $ is the vector of M\"{u}ntz-Jacobi functions with degree $(\mathcal{J}_i^{(0,q-1)}(u))\leq i \eta$, and $\underline{c}_j = [c_{j,0},c_{j,1},...,c_{j,N},0,...]$. Also, $\mathcal{J}$ is an invertible lower
triangular matrix, and $\underline{U}_u = [1,u^{\eta},\ldots,u^{N\eta}, ...]^{T}$. In addition, we suppose that
\begin{eqnarray}\label{e251}
\nonumber \bar{\psi}_{j}(u)&=&\sum_{i=0}^{\infty}\bar{\psi}_{j,i} ~u^{i\eta }=\underline{\bar{\psi}}_{j}\underline{U}_u,~~\underline{\bar{\psi}}_{j}=[\bar{\psi}_{j,0},\bar{\psi}_{j,1},...,\bar{\psi}_{j,N},0, ...],\\
\bar{p}_{j,r}(u)&=& \bar{p}_{j,r,N}(u)
=\sum_{i=0}^{\infty}p_{j,r,i} ~\mathcal{J}_i^{(0,q-1)}(u)\\
\nonumber &=&\underline{p}_{j,r} \mathcal{J} \underline{U}_u=\underline{\hat{p}}_{j,r} \underline{U}_u=\sum_{i=0}^{\infty}\hat{p}_{j,r,i} ~u^{i\eta },~~j=1,2,\ldots,n,~ r=1,2,\ldots,n+1,
\end{eqnarray}
in which $\underline{p}_{j,r}= [p_{j,r,0},p_{j,r,1},...,p_{j,r,N},0, ...]$, and $\underline{\hat{p}}_{j,r}= [\hat{p}_{j,r,0},\hat{p}_{j,r,1},...,\hat{p}_{j,r,N},0, ...]$. Substituting \eqref{eq11} and \eqref{e251} into \eqref{e367} yields
\begin{eqnarray}\label{e34}
\nonumber\underline{c}_j \mathcal{J} \underline{U}_u&=&\underline{\bar{\psi}}_{j}\underline{U}_u+T^{\theta_j}\sum\limits_{r=1}^{n}{\underline{c}_r \mathcal{J}\left(\sum_{i=0}^{\infty}\hat{p}_{j,r,i}~I^{\theta_j}\left(u^{i\eta}~\underline{U}_u\right)\right)}\\
&+&T^{\theta_j}\underline{\hat{p}}_{j,n+1}I^{\theta_j}\left(\underline{U}_u\right),\quad j=1,2,\ldots,n.
\end{eqnarray}
Therefore, it suffices to compute $I^{\theta_j}\left(u^{i\eta}~\underline{U}_u\right)$, and $I^{\theta_j}\left(\underline{U}_u\right)$. Consequently, by $I^{\alpha}u^{\beta}=frac{\Gamma(\beta+1)}{\Gamma(\alpha+\beta+1)}u^{\alpha+\beta}$, we arrive at
\begin{eqnarray}\label{eq15}
\nonumber &&I^{\theta_j}\left(u^{i\eta}~\underline{U}_u\right)=\left\{I^{\theta_j}\left(u^{(i+m)\eta}\right)\right\}_{m=0}^{\infty}=\left\{\xi_{j,i}^{m}~u^{(i+m)\eta+\theta_{j}}\right\}_{m=0 }^{\infty},\\
\nonumber\\
&&I^{\theta_j}\left(\underline{U}_u\right)=\left\{I^{\theta_j}\left(u^{m\eta}\right)\right\}_{m=0}^{\infty}=\left\{\vartheta_{j}^{m}~u^{m\eta+\theta_{j}}\right\}_{m=0 }^{\infty}=Q_{j} ~\underline{U}_u,
\end{eqnarray}
where $\xi_{j,i}^{m}=\frac{\Gamma((i+m)\eta+1)}{\Gamma((i+m)\eta+\theta_{j}+1)}$, $\vartheta_{j}^{m}=\frac{\Gamma(m\eta+1)}{\Gamma(m\eta+\theta_{j}+1)}$ and
\[
Q_j=
\begin{bmatrix}
\overbrace{0 \ldots  0}^{\theta_{j}q}&\vartheta_{j}^{0}&0&\ldots&\quad\\
\vdots&0&\vartheta_{j}^{1}&0&\cdots&\quad\\
\vdots&\vdots&0&\vartheta_{j}^{2}&0&\cdots \\
\cdots&\cdots&\quad&\ddots&\ddots&\ddots
\end{bmatrix}.
\]
For $j=1,2,\ldots,n$, inserting \eqref{eq15} into \eqref{e34} concludes
\begin{eqnarray}\label{e35}
\underline{c}_j \mathcal{J} \nonumber\underline{U}_u=\underline{\bar{\psi}}_{j}\underline{U}_u&+&T^{\theta_j}\sum\limits_{r=1}^{n}{\underline{c}_r \mathcal{J}\left(\sum_{i=0}^{\infty}\hat{p}_{j,r,i}~\left\{\xi_{j,i}^{m}~u^{(i+m)\eta+\theta_{j}}\right\}_{m=0}^{\infty}\right)}\\
&+&T^{\theta_j}\underline{\hat{p}}_{j,n+1}Q_{j} ~\underline{U}_u.
\end{eqnarray}
The relation \eqref{e35} can be rewritten as
\begin{equation}\label{e36}
\underline{c}_j \mathcal{J} \underline{U}_u=\underline{\bar{\psi}}_{j}\underline{U}_u+\sum\limits_{r=1}^{n}{\underline{c}_r \mathcal{J}A_{j,r}~\underline{U}_u}+T^{\theta_j}\underline{\hat{p}}_{j,n+1}Q_{j} ~\underline{U}_u,
\end{equation}
in which $A_{j,r},~j,r=1,2,\ldots,n$ are infinite upper-triangular matrices with the following entries
\begin{equation}\label{e38}
\big(A_{j,r}\big)_{k,\ell=0}^{\infty}=\begin{cases}
0,\hspace{+4.35cm} k \geq \ell-\theta_{j}q+1,\\
T^{\theta_j}\hat{p}_{j,r,\ell-k-\theta_{j}q}~\xi_{j,\ell-k-\theta_{j}q}^{k},\quad \quad k < \ell-\theta_{j}q+1.
\end{cases}
\end{equation}
Projecting \eqref{e36} onto $\{\mathcal{J}_i^{(0,q-1)}\}_{i =0}^{N}$ and using the orthogonality property of the M\"{u}ntz-Jacobi functions implies
\begin{equation}\label{e37}
\underline{c}_j^{N}=\bigg(\underline{\bar{\psi}}_{j}^{N}+\sum\limits_{r=1}^{n}{\underline{c}_r^{N} \mathcal{J}^{N}A_{j,r}^{N}}+T^{\theta_j}\underline{\hat{p}}_{j,n+1}^{N}Q_{j}^{N}\bigg)(\mathcal{J}^{N})^{-1}.
\end{equation}
The matrices and vectors with the index $N$ on top correspond to the sub-matrices and sub-vectors of order $N+1$, respectively. The unknown vectors are then obtained by solving the corresponding $n(N+1)\times n(N+1)$ system of algebraic equations \eqref{e37}.
\subsection{Numerical solvability and complexity analysis}
The main object of this section is to provide a theorem which reveals that the algebraic system \eqref{e37} has a unique solution and to propose a well-posed strategy to solve it.
\begin{theorem}
The linear algebraic system \eqref{e37} has a unique solution obtained by a determined recurrence relation.
\end{theorem}
\begin{proof}
Multiplying both sides of \eqref{e37} by $\mathcal{J}^{N}$, and defining
\begin{equation}\label{e41}
\underline{\tilde{c}}_j^{N}=\underline{c}_j^{N}\mathcal{J}^{N}=[\tilde{c}_{j,0},\tilde{c}_{j,1},...,\tilde{c}_{j,N}],
\end{equation}
we can rewrite \eqref{e37} as
\begin{equation}\label{e39}
\underline{\tilde{c}}_j^{ N}=\sum\limits_{r=1}^{n}{\underline{\tilde{c}}_r^{N} A_{j,r}^{N}}+\underline{P}_{j},\quad j=1, 2,...,n,
\end{equation}
in which
\begin{equation*}
 \underline{P}_{j}= \underline{\bar{\psi}}_{j}^{N}+T^{\theta_j}\underline{\hat{p}}_{j,n+1}^{N}Q_{j}^{N}=[P_{j,0},P_{j,1},\ldots,P_{j,N}].
\end{equation*}
Next, exploiting the structure of the upper-triangular matrix $(A_{j,r})_N$ defined in \eqref{e38}, we have
\begin{eqnarray*}
\underline{\tilde{c}}_r^{N} (A_{j,r})^{N}=\underline{\tilde{c}}_r^{N}\begin{bmatrix}
\overbrace{0 \ldots  0}^{\theta_{j}q}&(A_{j,r})_{0,\theta_{j}q}&(A_{j,r})_{0,\theta_{j}q+1}&\ldots&(A_{j,r})_{0,N}\\
\vdots&0&(A_{j,r})_{1,\theta_{j}q+1}&\cdots&(A_{j,r})_{1,N}\\
\vdots&\vdots&\vdots&\ddots&\ddots \\
0&0&\ddots&0&(A_{j,r})_{N-\theta_{j}q,N}\\
\vdots&\vdots&\vdots&\ddots&\vdots\\
0&0&0&0&0
\end{bmatrix}.
\end{eqnarray*}
From the relation above, it follows that
\begin{multline}\label{e40}
\big(\underline{\tilde{c}}_r^{N} (A_{j,r})^{N}\big)_{\ell=0}^{N}=\begin{cases}
0,\hspace{+3.3cm} \ell < \theta_{j}q,\\
\sum\limits_{k=0}^{\ell-\theta_{j}q}\tilde{c}_{r,k}(A_{j,r})_{k,\ell},\quad ~ \quad \ell \geq \theta_{j}q,
\end{cases}
\\
=[\overbrace{0 \ldots  0}^{\theta_{j}q},\bold{F}_{\theta_{j}q}^{j,r}(\tilde{c}_{r,0}),\bold{F}_{\theta_{j}q+1}^{j,r}(\tilde{c}_{r,0},\tilde{c}_{r,1}),\ldots,\bold{F}_{N}^{j,r}(\tilde{c}_{r,0},\tilde{c}_{r,1},\ldots,\tilde{c}_{r,N-\theta_{j}q})],
\end{multline}
where 
\[\bold{F}_{m}^{j,r}(\tilde{c}_{r,0},\tilde{c}_{r,1},\ldots,\tilde{c}_{r,m-\theta_{j}q}), \quad  \theta_{j}q \leq m \leq N,
\]
are linear functions of the elements $\tilde{c}_{r,0},~\tilde{c}_{r,1},\ldots,~\tilde{c}_{r,m-\theta_{j}q}$.
Inserting \eqref{e40} into \eqref{e39}, it can be concluded that the unknown components of the unknown vectors $\{\underline{\tilde{c}}_j^{N}\}_{j=1}^{n}$ are calculated by the following recurrence relations
\begin{eqnarray}\label{Recurence}
\nonumber\tilde{c}_{j,0}&=&P_{j,0},...,\tilde{c}_{j,\theta_{j}q-1}=P_{j,\theta_{j}q-1},\\
\nonumber\tilde{c}_{j,\theta_{j}q}&=&\sum\limits_{r=1}^{n}\bold{F}_{\theta_{j}q}^{j,r}(\tilde{c}_{r,0})+P_{j,\theta_{j}q},\\
\nonumber\vdots &~&\\
\tilde{c}_{j,N}&=&\sum\limits_{r=1}^{n}\bold{F}_{N}^{j,r}(\tilde{c}_{r,0},\tilde{c}_{r,1},\ldots,\tilde{c}_{r,N-\theta_{j}q})+P_{j,N},
\end{eqnarray}
which implies the result.
\end{proof}

In fact, the direct solution of the $n(N+ 1)\times n(N+ 1)$ system of algebraic equations \eqref{e37} can yield inaccurate approximations, due to the increased computational cost and the possibility of increasing the condition
number of the coefficient matrix, especially for a large degree of approximation $N$. Accordingly, we avoid solving the system of algebraic equations \eqref{e37} directly, and compute the unknowns by recurrence relations.
\\
Finally, upon solving the lower triangular system \eqref{e41} to determine $\{\underline{c}_j^{N}\}_{j=1}^{n}$, the corresponding M\"{u}ntz-Jacobi Galerkin solutions \eqref{eq11} are obtained. Then the approximate
solutions $v_{j,N}(t)$ of the main problem \eqref{eq1} are given by
\begin{equation*}
v_{j,N}(t)=\bar{v}_{j,N}(\frac{t}{T}),\quad j=1,2,\ldots,n.
\end{equation*}
\section{Error analysis}\label{sec4}
This section is devoted to analyzing the convergence of the proposed method by deriving an appropriate error bound in a weighted $L^{2}$-norm.
\begin{theorem}[Error estimate]\label{thm4}
Suppose that $\bar{v}_{j,N}(u)$ given by \eqref{eq11} are the approximate solutions of \eqref{e367}. If the following conditions satisfy
\begin{enumerate}

 \item  $D_u^{\epsilon_{jr}+1}\big(I^{\theta_{j}}\big(\bar{p}_{j,r}(u)\bar{v}_{r}(u)\big)\big) \in C(I),$

    \item $I^{\theta_{j}}\big(\bar{p}_{j,r}(u)\bar{v}_{r}(u)\big) \in H_{\hat{w}^{(0,0)}}^{\epsilon_{jr}}(I),~~ \quad \epsilon_{jr} \geq 0,$

    \item $I^{\theta_{j}}\big(\bar{p}_{j,n+1}(u)\big) \in H_{\hat{w}^{(0,0)}}^{\rho_{j}}(I),~~\quad ~\quad \rho_{j} \geq 0,$

\end{enumerate}
then for sufficiently large values of $N$ we have
\begin{align*}
\|\bar{e}_{j,N}(u)\|_{\hat{w}^{(0,0)}} \leq C \Bigg(N^{-\epsilon_{jr}}|I^{\theta_{j}}\big(\bar{p}_{j,r}(u)\bar{v}_{r}(u)\big)|_{\epsilon_{jr},\hat{w}^{(0,0)}}+N^{-\rho_{j}}|I^{\theta_{j}}\big(\bar{p}_{j,n+1}(u)\big)|_{\rho_j,\hat{w}^{(0,0)}}\Bigg),
\end{align*}
where $\bar{e}_{j,N}(u)=\bar{v}_{j}(u)-\bar{v}_{j,N}(u)$ denotes the error function, and $C>0$ is a constant independent of $N$.
\end{theorem}

\begin{proof}
According to the proposed numerical scheme in Section \ref{sec1}, we have
\begin{equation}\label{e50a1}
\bar{v}_{j,N}(u)=\bar{\psi}_{j}(u)+T^{\theta_j}\sum\limits_{r=1}^{n}{I^{\theta_j}\big(\bar{p}_{j,r}(u)\bar{v}_{r,N}(u)\big)}+T^{\theta_j}I^{\theta_j}\big(\bar{p}_{j,n+1}(u)\big),\quad j =1,2,\ldots,n.
\end{equation}
Then, following the same approach used to derive \eqref{e37}, we multiply both sides of \eqref{e50a1} by $\mathcal{J}_i^{(0,q-1)}(u)$ and integrate over $I$, which yields
\begin{equation}\label{e50a2}
\big( \bar{v}_{j,N}-\bar{\psi}_{j}-T^{\theta_j}\sum\limits_{r=1}^{n}{I^{\theta_j}\big(\bar{p}_{j,r}\bar{v}_{r,N}\big)}-T^{\theta_j}I^{\theta_j}\big(\bar{p}_{j,n+1}\big),\mathcal{J}_i^{(0,q-1)}\big)=0, \quad i=0,1,\ldots,N.
\end{equation}
Now, multiplying both sides of \eqref{e50a2} by $\{\frac{\mathcal{J}_i^{(0,q-1)}(u)}{\hat{\zeta}_i^{(0,q-1)}}\}_{i=0}^{N}$ and summing up from $i=0$ to $i=N$, one arrives at
\begin{equation*}
\pi_{N}\big( \bar{v}_{j,N}-\bar{\psi}_{j}-T^{\theta_j}\sum\limits_{r=1}^{n}{I^{\theta_j}\big(\bar{p}_{j,r}\bar{v}_{r,N}\big)}-T^{\theta_j}I^{\theta_j}\big(\bar{p}_{j,n+1}\big)\big)=0,
\end{equation*}
which is equivalent to
\begin{equation}\label{e50}
\bar{v}_{j,N}(u)=\bar{\psi}_{j}(u)+T^{\theta_j}\sum\limits_{r=1}^{n}\pi_{N}\big( {I^{\theta_j}\big(\bar{p}_{j,r}\bar{v}_{r,N}\big)}\big)+T^{\theta_j}\pi_{N}\big( I^{\theta_j}\big(\bar{p}_{j,n+1}\big)\big),
\end{equation}
since we have \( \bar{v}_{j,N}(u),\ \bar{\psi}_j(u) \in \operatorname{span}\{ \mathcal{J}_0^{(0,q-1)},\ \mathcal{J}_1^{(0,q-1)},\ \ldots,\ \mathcal{J}_N^{(0,q-1)} \}. \) Subtracting \eqref{e50} from \eqref{e367} yields
\begin{eqnarray*}
\bar{e}_{j,N}(u)=T^{\theta_j}\sum\limits_{r=1}^{n} {I^{\theta_j}\big(\bar{p}_{j,r}(u)\bar{v}_{r}(u)}\big)&-&T^{\theta_j}\sum\limits_{r=1}^{n}\pi_{N}\big( {I^{\theta_j}\big(\bar{p}_{j,r}(u)\bar{v}_{r,N}(u)\big)}\big)\\
&+&T^{\theta_j} I^{\theta_j}\big(\bar{p}_{j,n+1}(u)\big)-T^{\theta_j}\pi_{N}\big( I^{\theta_j}\big(\bar{p}_{j,n+1}(u)\big)\big),\hspace{+2.3cm}
\end{eqnarray*}
which can be rewritten in the following sense
\begin{equation}\label{Eanalysis}
\bar{e}_{j,N}(u)=T^{\theta_j}\sum\limits_{r=1}^{n}{I^{\gamma_{j}}\big(\bar{p}_{j,r}(u)\bar{e}_{r,N}(u)\big)}+\Pi_{j}.
\end{equation}
Equivalently we obtain
\begin{equation*}
\bar{e}_{j,N}(u)=\sum\limits_{r=1}^{n}\Bigg(\int_{0}^{u} k_{j,r}(u,w)\bar{e}_{r,N}(w)dw\Bigg)+\Pi_{j},
\end{equation*}
where $k_{j,r}(u,w)=\frac{T^{\theta_{j}}}{\Gamma(\theta_{j})}(u-w)^{\theta_{j}-1}\bar{p}_{j,r}(u)$, and
\begin{equation*}
    \Pi_{j}=T^{\theta_j}\sum\limits_{r=1}^{n} {e_{\pi_{N}}\Big(I^{\theta_j}\big(\bar{p}_{j,r}(u)\bar{v}_{r,N}(u)}\big)\Big)+T^{\theta_j}e_{\pi_{N}}\Big( I^{\theta_j}\big(\bar{p}_{j,n+1}(u)\big)\Big).
\end{equation*}
Here, $e_{\pi_{N}} (f)=f-\pi_{N}(f)$. Defining the vectors
\begin{align*}
\underline{E}(u)=\big[\bar{e}_{1,N}(u),\bar{e}_{2,N}(u),\ldots,\bar{e}_{n,N}(u)\big]^{T},\quad \quad \underline{\Pi}=\big[ \Pi_{1},\Pi_{2},\ldots,\Pi_{n}\big]^T,
\end{align*}
we can rewrite the equation \eqref{Eanalysis} as the following matrix formulation
\begin{align}\label{eq22}
\underline{E}(u)=\int_{0}^{u}(u-w)^{\min\limits_{1 \le l \le n}\lbrace\theta_{l}\rbrace-1}K(u,w)\underline{E}(w)dw+\underline{\Pi},
\end{align}
where
\begin{eqnarray*}
K(u,w)&=&\bigg((u-w)^{1-\min\limits_{1 \le l \le n} \lbrace \theta_{l}\rbrace}k_{j,r}(u,w)\bigg)_{j,r=1}^{n}\\
\\
&=&\bigg(\frac{T^{\theta_{j}}}{\Gamma(\theta_{j})}(u-w)^{\theta_{j}-\min\limits_{1 \le l \le n} \lbrace \theta_{l}\rbrace}\bar{p}_{j,r}(u)\bigg)_{j,r=1}^{n},
\end{eqnarray*}
is a continuous function on $\{(u,w):~ 0 \leq w \leq u \leq 1\}$. From \eqref{eq22}, it follows that
\begin{align*}
\vert \underline{E}\vert \leq \Psi\int_{0}^{u}(u-w)^{\min\limits_{1 \le l \le n}\lbrace\theta_{l}\rbrace-1}\vert \underline{E}(w)\vert dw+\vert \underline{\Pi}\vert,
\end{align*}
where $\Psi=\max\limits_{0 \leq w \leq u \leq 1}{\vert K(u,w)\vert}< \infty$. Applying Gronwall's inequality, i.e., Lemma 6 of \cite{35} leads to
\begin{align*}
\|\underline{E}\|_{\hat{w}^{(0,0)}} \leq C \left\|\underline{\Pi}\right\|_{\hat{w}^{(0,0)}},
\end{align*}
and consequently
\begin{eqnarray}\label{e54}
\nonumber\Vert\bar{e}_{j,N}(u)\Vert_{\hat{w}^{(0,0)}} &\leq& C \Vert\Pi_{j}\Vert_{\hat{w}^{(0,0)}}\\
\nonumber&\leq& C \Bigg(\sum\limits_{r=1}^{n} {\|e_{\pi_{N}}\Big(I^{\theta_j}\big(\bar{p}_{j,r}(u)\bar{v}_{r,N}(u)\big)\Big)\|_{\hat{w}^{(0,0)}}}
\\
&+&\|e_{\pi_{N}}\Big( I^{\theta_j}\big(\bar{p}_{j,n+1}(u)\big)\Big)\|_{\hat{w}^{(0,0)}}\Bigg).
\end{eqnarray}
Using Theorem \ref{thm2}, we deduce
\begin{eqnarray}\label{e55}
\nonumber\|e_{\pi_{N}}\Big(I^{\theta_j}\big(\bar{p}_{j,r}(u)\bar{v}_{r,N}(u)\big)\Big)\|_{\hat{w}^{(0,0)}}&\leq& C N^{-\epsilon_{jr}} |I^{\theta_{j}}\big(\bar{p}_{j,r}(u)\bar{v}_{r,N}(u)\big)|_{\epsilon_j,\hat{w}^{(0,0)}}\\
\nonumber&\leq& C N^{-\epsilon_{jr}} \Big(|I^{\theta_{j}}\big(\bar{p}_{j,r}(u)\bar{v}_{r}(u)\big)|_{\epsilon_{jr},\hat{w}^{(0,0)}}\\
\nonumber&+&\|D_u\big(D_u^{\epsilon_{jr}} I^{\theta_{j}}\big(\bar{p}_{j,r}(u)\bar{v}_{r}(u)\big)\big)\|_\infty \|\bar{e}_{r,N} \|_{\hat{w}^{(0,0)}}\Big)\\
&\leq& C N^{-\epsilon_{jr}} \Big(|I^{\theta_{j}}\big(\bar{p}_{j,r}(u)\bar{v}_{r}(u)\big)|_{\epsilon_{jr},\hat{w}^{(0,0)}}+\|\bar{e}_{r,N} \|_{\hat{w}^{(0,0)}}\Big),
\end{eqnarray}
where follows from the use of the first assumption together with the first order Taylor formula. Again from Theorem \ref{thm2}, we derive
\begin{eqnarray}\label{e56}
\|e_{\pi_{N}}\Big( I^{\theta_j}\big(\bar{p}_{j,n+1}(u)\big)\Big)\|_{\hat{w}^{(0,0)}}&\leq& C N^{-\rho_j} |I^{\theta_{j}}\big(\bar{p}_{j,n+1}(u)\big)|_{\rho_j,\hat{w}^{(0,0)}}.  
\end{eqnarray}
Inserting \eqref{e55} and \eqref{e56} into \eqref{e54} gives
\begin{equation}\label{e57}
\Vert \bar{e}_{j,N}(u)\Vert_{\hat{w}^{(0,0)}} -C N^{-\epsilon_{jr}} \sum_{r=1}^{n}\| \bar{e}_{r,N}(u)\|_{\hat{w}^{(0,0)}}\leq C H_{j},
\end{equation}
in which
\begin{equation*}
H_{j}=N^{-\epsilon_{jr}}|I^{\theta_{j}}\big(\bar{p}_{j,r}(u)\bar{v}_{r}(u)\big)|_{\epsilon_{jr},\hat{w}^{(0,0)}}+N^{-\rho_{j}}|I^{\theta_{j}}\big(\bar{p}_{j,n+1}(u)\big)|_{\rho_j,\hat{w}^{(0,0)}}.
\end{equation*}
It is straightforward to rewrite \eqref{e57} in the vector-matrix form
\begin{equation}\label{e58}
  G \underline{e} \leq C \underline{H},
\end{equation}
where
\begin{eqnarray*}
\underline{e}&=&[\|\bar{e}_{1,N}\|_{\hat{w}^{(0,0)}}, \|\bar{e}_{2,N}\|_{\hat{w}^{(0,0)}},...,\|\bar{e}_{n,N}\|_{\hat{w}^{(0,0)}}]^{T},\\
\underline{H}&=&[H_1, H_2, ..., H_n]^{T},
\end{eqnarray*}
and $G$ is a matrix of order $n$ with the entries
\[
\big(G\big)_{j,r=1}^{n}=\begin{cases}
1-CN^{-\epsilon_{jj}},\quad r=j,\\
-CN^{-\epsilon_{jr}},\quad \quad r \neq j.
\end{cases}
\]
It follows that, for sufficiently large $N$, the matrix $G$ converges to the identity matrix. Therefore, the inequality \eqref{e58} yields
\[
\|\bar{e}_{j,N}\|_{\hat{w}^{(0,0)}} \leq C H_{j},
\]
which proves the desired result.
\end{proof}
\begin{remark}
It should be emphasized that $D_u^k$ is not the classical derivative operator; rather, it is the derivative defined in \eqref{Newd} for the transformed space $H_{\hat{w}^{(0,0)}}^k$. Under the assumptions of Theorem \ref{th202}, the solutions $\bar{v}_{j}(u)$ are analytic in powers of $u^{\eta}$, which implies that $\{\epsilon_{jr}\}_{j=1}^{n} = \{\rho_{j}\}_{j=1}^{n} = \{\infty\}_{j=1}^{n}$. Therefore, the proposed method is expected to achieve spectral accuracy in this setting. Although the algorithm can also be applied to equations that do not fully satisfy these assumptions, some conclusions may no longer hold. For all the test problems considered (see Section \ref{sec5}), the assumptions of Theorem \ref{th202} are indeed satisfied.
\end{remark}
\section{Numerical examples}\label{sec5}
In this section, we present numerical solutions for several SFDEs to demonstrate the accuracy and efficiency of the proposed scheme. The section is structured as follows.

\begin{itemize}
\item The numerical errors $E(N)$ and CPU time\footnote{We used the \texttt{Timing} function in Mathematica to report CPU time.} used are computed to evaluate the performance of the suggested method. For all examples, these errors are measured using the $L^{2}(\Lambda)$-norm
    \[
    E(N)=\max\limits_{j=1,2,\ldots,n}\|e_{j,N}(t)\|, \quad \quad \|e_{j,N}\|^{2}\simeq\frac{T}{2}\sum_{k=0}^{N}\bigg(e_{j,N}\big(T/2(x_{k}^{(0,0)}+1)\big)\bigg)^{2}w_{k}^{(0,0)},\]
    where $e_{j,N}(t)=v_{j}(t)-v_{j,N}(t)$, and $\{(x_{k}^{(0,0)},w_{k}^{(0,0)})\}_{k=0}^{N}$ denote the Legendre-Gauss quadrature node-weight pairs \cite{29}.
\item To demonstrate the capability of our implementation in handling highly oscillatory solutions of SFDEs, we compare its performance with existing methods from the literature.

\item  To assess the stability of the proposed approach, we examine its performance for large values of \(N\) and for problems defined over an extended integration domain \(\Lambda\).
\end{itemize}
All the calculations were executed using Mathematica v13.3, running on an Intel (R) Core (TM) i7-8665U CPU @ 1.90 GHz with 16 GB of RAM.
\begin{example}\label{exm1}
Consider the highly oscillatory linear SFDEs 
\begin{align*}
\begin{cases}
D^{\theta_{1}}_{C} v_1(t)= t^{1/2} v_1(t)+ v_2(t) +\frac{1}{2} J_{0}(t^{\frac{5}{4}}) v_3(t) + p_{1,4}(t), \\
D^{\theta_{2}}_{C} v_2(t)= v_{1}(t)+t v_2(t) + 2t^{\frac{3}{2}} v_{3}(t) + p_{2,4}(t), \\
D^{\theta_{3}}_{C} v_{3}(t)= \sin{(2t^{\frac{1}{2}})}v_{1}(t)+3v_{2}(t)+ t v_{3}(t) + p_{3,4}(t),\\
v_1(0)=0,\quad v_2(0)=v_{3}(0)=1, \quad  t \in [0,\frac{\pi}{2}],
\end{cases}
\end{align*}
where \( J_{c}(t) \) denotes the Bessel function for the integer \( c \), \( \theta_{1} = \frac{1}{4} \), \( \theta_{2} = \frac{1}{2} \), and \( \theta_{3} = \frac{3}{4} \), and the source functions \( p_{j,4}(t) \), \( j = 1, 2, 3 \), are chosen such that the exact solutions are given by
\begin{align*}
v_1(t)=\sin{(70t^{\frac{1}{4}})},\quad v_2(t)=\cos{(70t^{\frac{1}{2}})},\quad v_3(t)=\sin{(70t^{\frac{3}{4}})}+\cos{(12t^{\frac{3}{4}})}.
\end{align*}
\end{example}

This problem exhibits a highly oscillatory solution in which $v_{1}(t)$ and $v_{3}(t)$ are not infinitely smooth, and their first derivatives are discontinuous at $t = 0$. Such behavior makes the problem challenging, as its numerical approximation may become unstable for large $N$. We solve this problem using the proposed method, and the obtained results are presented in Figures \ref{fg4550}, \ref{fg2500}, and Table \ref{tab_cpu1}.
\begin{figure}[!ht]
\centering
\begin{minipage}{0.48\textwidth}
    \centering
    \includegraphics[width=\linewidth]{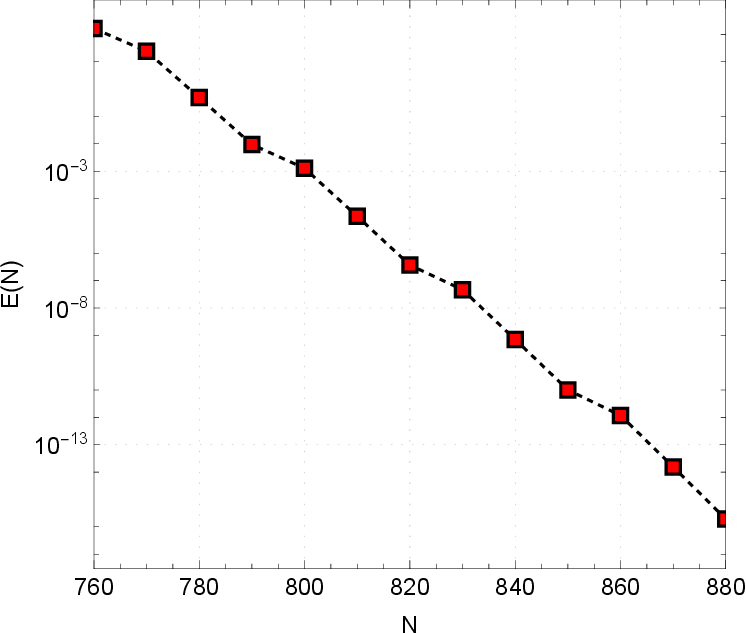}
    \caption{Semi-logarithmic plot of the numerical errors versus $N$ for Example \ref{exm1}.}
    \label{fg4550}
\end{minipage}\hfill
\begin{minipage}{0.48\textwidth}
    \centering
   
    \vspace{3.7cm}
      \renewcommand{\arraystretch}{1.3}
    \begin{tabular}{c|c}
        $N$ & CPU Time (s) \\
        \hline
        110 & 3.34 \\
        220 & $1.58\times 10^{+1}$ \\
        440 & $6.30\times 10^{+1}$ \\
        880 & $5.41\times 10^{+2}$\\
    \end{tabular}
    \captionof{table}{Elapsed CPU time versus $N$ for Example \ref{exm1}.}
    \label{tab_cpu1}
\end{minipage}
\end{figure}
\begin{figure}[!ht]
\centerline{\includegraphics[width=8cm]{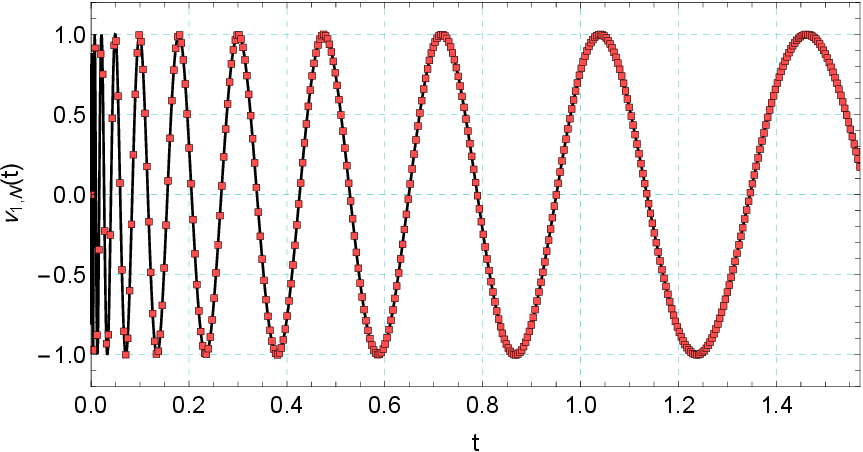},\includegraphics[width=8cm]{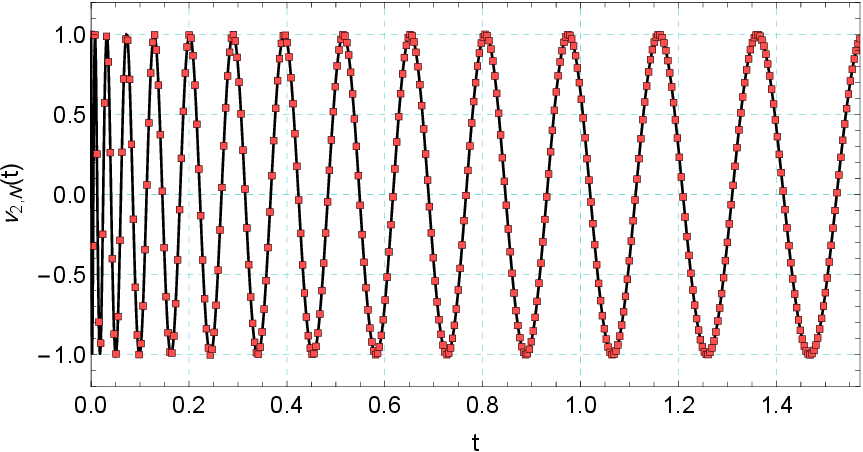}}\centerline{\includegraphics[width=8cm]{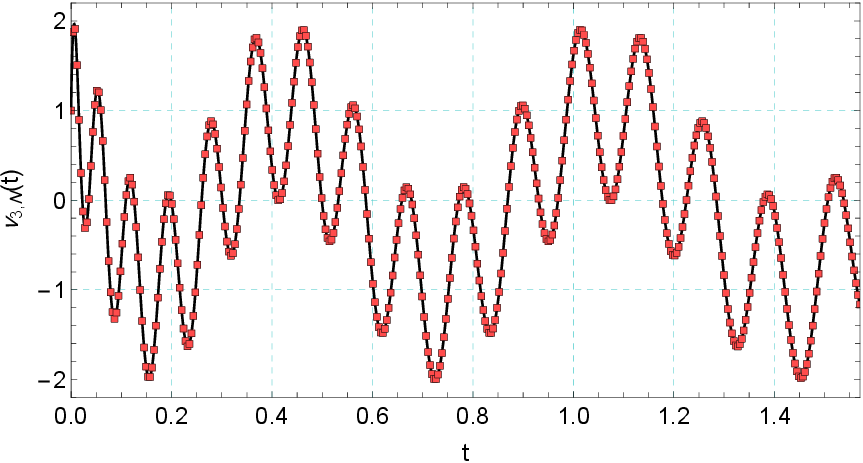}}
\caption{Plots comparing the approximate solutions at $N=1000$ (red squares) and the exact solutions (solid lines) for Example \ref{exm1}.}\label{fg2500}
\end{figure}

 From Figure \ref{fg4550}, we can deduce that the method is converging for $N > 750$, and due to its stability, the error consistently decreases as $N$ increases. Moreover, the semi-log plot of the errors versus $N$ shows a nearly linear trend, indicating an exponential rate of convergence (note that we have
$\{\epsilon_{jr}\}_{j=1}^{3}=\{\rho_{j}\}_{j=1}^{3}=\{\infty\}_{j=1}^{n}$ in Theorem \ref{thm4}). Figure \ref{fg2500} shows that the approximations exhibit excellent agreement with the exact solution curves, demonstrating the high accuracy of the method. The results validate the method’s ability to handle challenging problems with oscillatory behavior and discontinuous derivatives, maintaining both accuracy and stability.
\begin{example}\label{exm3}
Consider the highly oscillatory linear SFDEs 
\begin{align*}
\begin{cases}
D^{\theta_{1}}_{C} v_1(t)= t^{5/2} v_1(t)+ v_2(t) + p_{1,3}(t), \\
D^{\theta_{2}}_{C} v_2(t)= v_{1}(t)+\cos{(t^{3/2})} v_2(t) + p_{2,3}(t),\\
v_1(0)=0,\quad v_2(0)=1, \quad v_2^{(1)}(0)=1, \quad \Lambda= [0,\frac{3\pi}{2}].
\end{cases}
\end{align*}
Here we take \( \theta_{1} = \frac{1}{2} \), \( \theta_{2} = \frac{3}{2} \), and the functions \( p_{j,3}(t) \), \( j = 1, 2 \), are selected so that the problem admits the following exact solutions
\begin{align*}
v_1(t)=t^{\frac{1}{2}} e^{i80t^{\frac{1}{2}}},\quad v_2(t)=e^{i10t^{\frac{3}{2}}}.
\end{align*}
\end{example}
The problem is solved using the proposed method, and the corresponding results are shown in Figure \ref{fg31}, Figure \ref{fg3}, and Table \ref{tab_cpu2}.
\begin{figure}[!ht]
\centering
\begin{minipage}{0.48\textwidth}
    \centering
    \includegraphics[width=\linewidth]{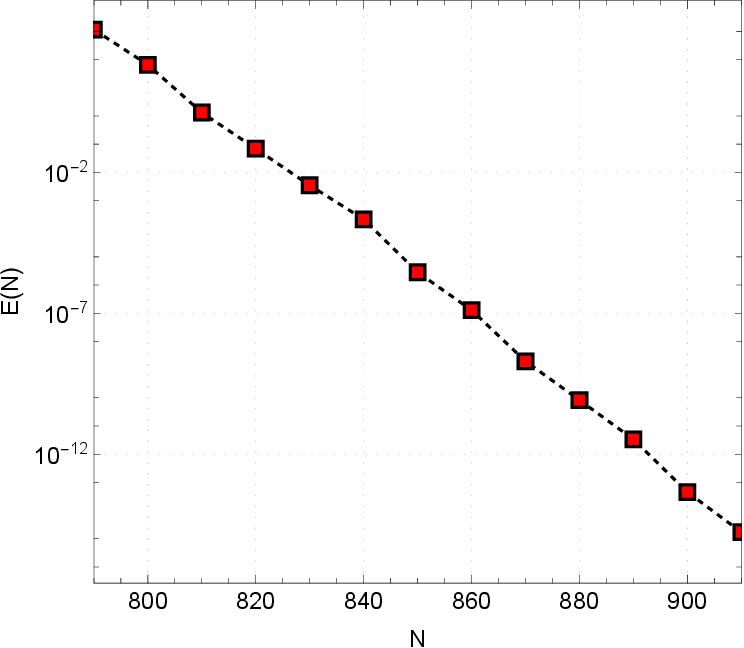}
    \caption{Semi-logarithmic plot of the numerical errors versus $N$ for Example \ref{exm3}.}
    \label{fg31}
\end{minipage}\hfill
\begin{minipage}{0.48\textwidth}
    \centering

    \vspace{3.75cm} 
      \renewcommand{\arraystretch}{1.3}
    \begin{tabular}{c|c}
        $N$ & CPU Time (s) \\
        \hline
        125 & 2.78 \\
        250 & $1.61\times 10^{+1}$ \\
        500 & $5.61\times 10^{+1}$ \\
        1000 & $4.32\times 10^{+2}$\\
    \end{tabular}
    \captionof{table}{Elapsed CPU time versus $N$ for Example \ref{exm3}.}
    \label{tab_cpu2}
\end{minipage}
\end{figure}
\begin{figure}[!ht]
\centerline{\includegraphics[width=8cm]{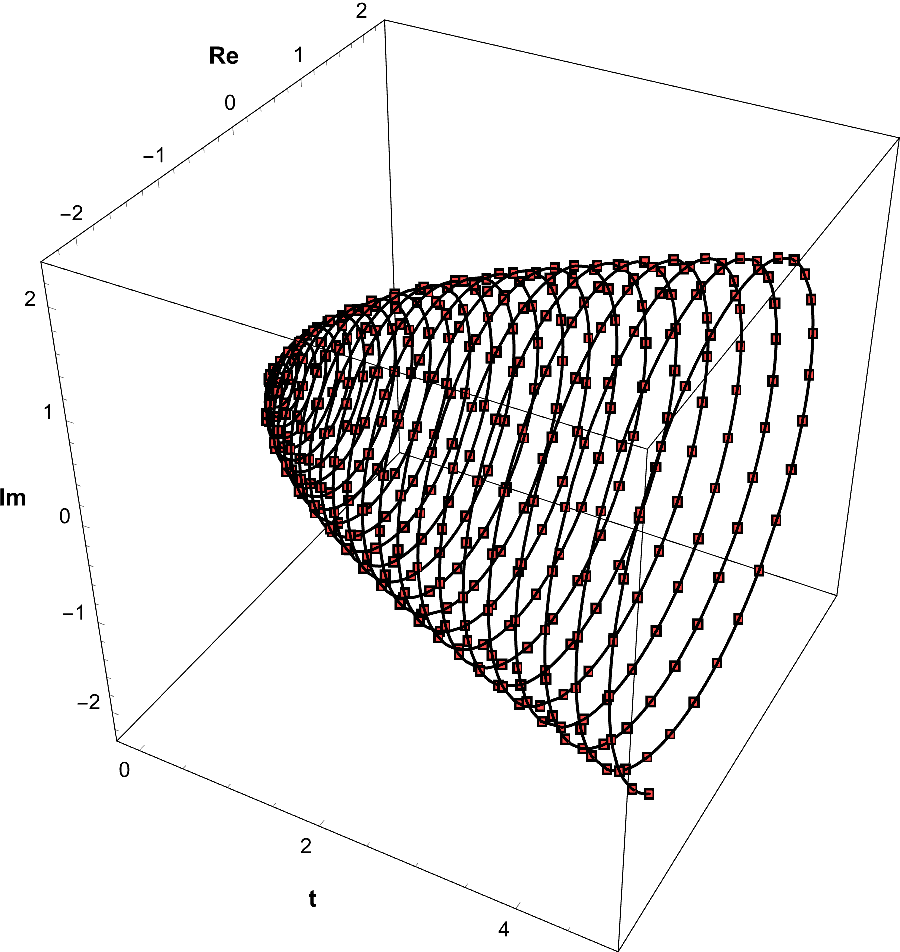},\includegraphics[width=8cm]{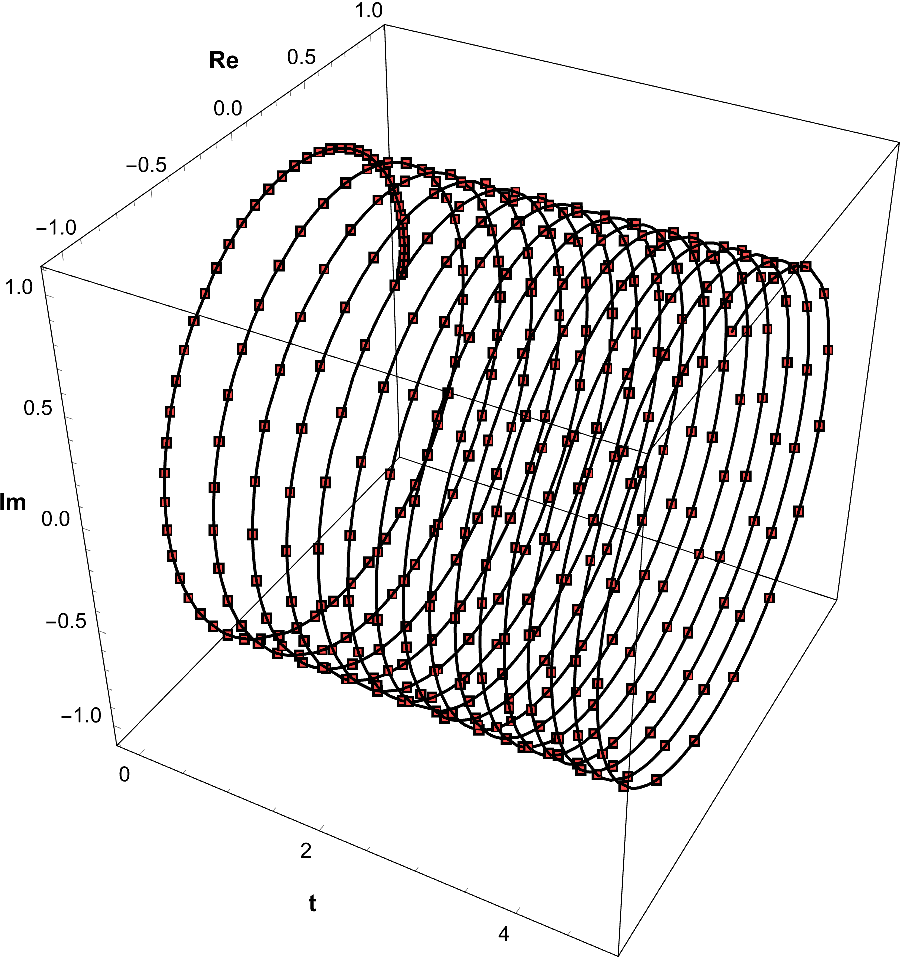}}
\caption{Plots comparing the approximate solutions at $N=1200$ (red squares) and the exact solutions (solid lines) for Example \ref{exm3}.}\label{fg3}
\end{figure}

The forcing functions $p_{j,3}(t)$, $j=1,2$, are found to have the following asymptotic behaviors
\begin{eqnarray*}
    p_{1,3}(t)&=&-(1.14 \times 10^{-1})+(9.03 \times 10^{1} i)t^{1/2}-(4.25 \times 10^{3}) t-(1.28 \times 10^{5} i) t^{3/2}+\cdots,\\
    p_{1,3}(t)&=&-(1.00-1.33 \times 10^{1} i)-t^{1/2}+(8.00 \times 10^{1}i) t+(2.97 \times 10^{3}-1.00 \times 10^{1}i) t^{3/2}+\cdots.
\end{eqnarray*}
Consequently, $v_{j}(s^{q})$, and $q_{j,r} (s^{q})$ for $q=2$ are analytic functions and thereby we have $\{\epsilon_{jr}\}_{j=1}^{2}=\{\rho_{j}\}_{j=1}^{2}=\{\infty\}_{j=1}^{n}$ in Theorem \ref{thm4}. This establishes the exponential-type convergence of the approximate solutions, as illustrated in Figure \ref{fg31}, while the semi-logarithmic plot of the numerical errors exhibits an almost linear dependence on $N$. Moreover, Figure \ref{fg3} shows that the approximate solutions coincide closely with the exact solutions at the selected nodes, confirming the reliability and high-order accuracy of the approximations.

\begin{example}\label{exm2}
Consider the following SFDEs
\begin{align}\label{eqq45}
\begin{cases}
D^{\theta_{1}}_{C} v_1(t)= 2t v_1(t)+t^{\frac{1}{3}} v_2(t) + \sin{(2t^{\frac{1}{6}})} v_3(t) + p_{1,4}(t), \\
D^{\theta_{2}}_{C} v_2(t)=t^{\frac{11}{6}} v_{1}(t)+t^{\frac{1}{2}} v_2(t) + 5 v_{3}(t) + p_{2,4}(t), \\
D^{\theta_{3}}_{C} v_{3}(t)= tv_{1}(t)+v_{2}(t)+ \cos{(t^{\frac{2}{3}})} v_{3}(t) + p_{3,4}(t),\\
v_1(0)= v_2(0)=v_{3}(0)=0, \quad  \Lambda=  [0,1],
\end{cases}
\end{align}
with the non-smooth solutions
\begin{align*}
v_1(t)=\sin{(10t^{\frac{1}{6}})},\quad v_2(t)=t^{\frac{1}{3}},\quad v_3(t)=t^{\frac{2}{3}}+5t^{\frac{5}{6}}.
\end{align*}
Given the exact solutions above, the forcing functions $p_{j,4}(t)$, $j=1,2,3$, are determined so that these solutions satisfy \eqref{eqq45}.
\end{example}
We solve this problem using our proposed method and the fractional collocation method introduced by Faghih and Mokhtary \cite{36}, and the corresponding semi-logarithmic plots of the errors are presented in Figure \ref{fg2}. Table \ref{tab_cpu3} reports the elapsed CPU times for both methods.
\begin{figure}[!ht]
\centering
\begin{minipage}{0.48\textwidth}
    \centering
    \includegraphics[width=\linewidth]{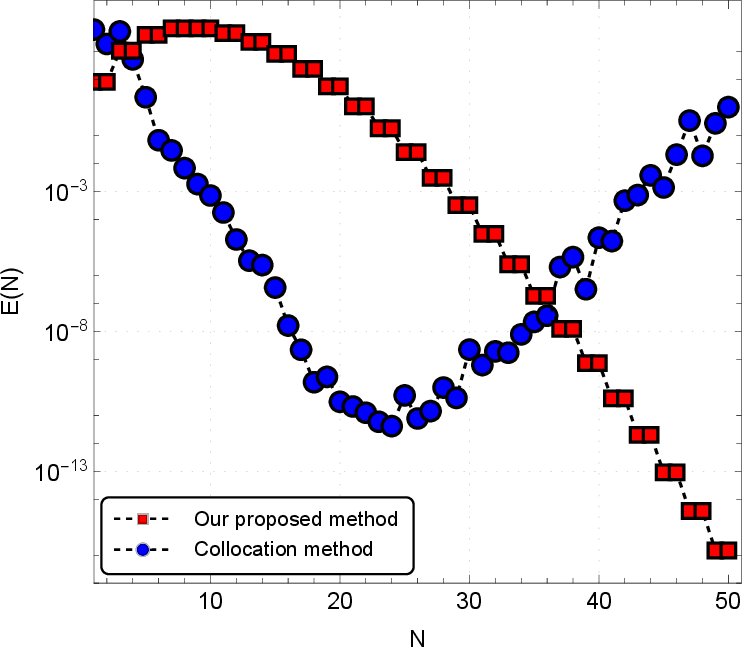}
    \caption{Semi-logarithmic plot of the numerical errors versus $N$ for Example \ref{exm2}.}
    \label{fg2}
\end{minipage}\hfill
\begin{minipage}{0.48\textwidth}
    \centering
    \vspace{2.8cm} 
    \renewcommand{\arraystretch}{1.3}
    \setlength{\tabcolsep}{4pt}
    \begin{tabular}{c|c|c}
        \multirow{2}{*}{$N$} & \multicolumn{2}{c}{CPU Time (s)} \\
        \cline{2-3}
         & Proposed method & Collocation method \\
        \hline
        10  & $1.09\times 10^{-1}$           & $6.25\times 10^{-1}$\\
        20  & $1.41\times 10^{-1}$ & 1.88 \\
        40  & $2.81\times 10^{-1}$ & $2.16\times 10^{+1}$ \\
        80 & $8.13\times 10^{-1}$ & $1.95\times 10^{+2}$ \\
    \end{tabular}
    \captionof{table}{Comparison of elapsed CPU times between our proposed method and the collocation method in \cite{36} for Example \ref{exm2}.}
    \label{tab_cpu3}
\end{minipage}
\end{figure}

In collocation methods, the residual equation is evaluated at a selected set of collocation nodes, leading to a linear system of algebraic equations that, in most cases, has a dense coefficient matrix \cite{36}. Solving such systems can result in reduced accuracy as shown in Figure \ref{fg2}, particularly for oscillatory solutions, where achieving high accuracy often requires large approximation degrees. From Figure \ref{fg2}, we observe that although the collocation method converges more rapidly than our proposed method for small $N$, the error reaches a lower bound around $10^{-11}$, beyond which increasing $N$ no longer improves accuracy. Higher approximation degrees lead to greater accumulation of rounding errors and a more ill-conditioned dense algebraic system (e.g., for $N=30$, the condition number of the resulting algebraic system is around $10^7$). Consequently, for $N \geq 25$, the error grows sharply.

In contrast, the proposed M\"{u}ntz–Jacobi Galerkin algorithm yields significantly better approximations with considerably shorter CPU times. A regular and steady decrease in the error is observed, demonstrating the advantage of our approach in producing high-accuracy approximations for oscillatory and less-smooth solutions of SFDEs. This advantage partly arises from computing the approximate solutions via recurrence formulas \eqref{Recurence}, which avoids the formation and solution of large, and potentially ill-conditioned algebraic systems.
\\

In the following, we compare our method with the central part interpolation approach \cite{Lillemae2025} and the hybrid non-polynomial collocation method \cite{Ferras2020}.
\begin{example}\label{exm5}
     \cite[Example 2]{Ferras2020} Let us consider the linear problem
\begin{equation*}
\begin{cases}
D^{\theta}_Cv_{1}(t)=v_{2}(t),\\
D^{\theta}_Cv_{2}(t)=-v_{1}(t)-v_{2}(t)+t^{\theta+1}+\frac{\pi \csc(\pi \theta)t^{1-\theta}}{\Gamma(-\theta-1)\Gamma(2-\theta)}+\frac{\pi t \csc(\pi \theta)}{\Gamma(-\theta-1)},\\
v_{1}(0)=0, \quad v_{2}(0)=0, \quad \Lambda=[0,1],
\end{cases}
\end{equation*}
along with the exact solutions
\begin{align*}
v_1(t)=t^{1+\theta},\quad v_2(t)=\frac{\pi \theta (\theta+1) \csc(\pi \theta)}{\Gamma(1-\theta)}t.
\end{align*}
\end{example}
We solve this problem using the proposed scheme for $\theta = \tfrac{1}{4}, \tfrac{2}{5}, \tfrac{1}{2}, \tfrac{2}{3}$, obtaining exact solutions up to machine precision with approximation degrees $N = 5, 7, 3, 5$, respectively, in a very short CPU time (less than $5.0\times 10^{-2}$ (s)). This problem was addressed by Ferrás et al. \cite{Ferras2020} using a hybrid collocation approach for derivative values $\theta = \tfrac{1}{4}, \tfrac{1}{2}, \tfrac{2}{3}$. This method involves splitting the domain $\Lambda$ into $m$ subintervals, representing the approximate solutions as a linear combination of non-polynomial functions near the origin, and using polynomials in the remainder of the domain. For the case $\theta = \tfrac{2}{5}$, Lillemäe et al. \cite{Lillemae2025} proposed a collocation method based on central part interpolation, utilizing continuous piecewise polynomials on a uniform grid. This method involves constructing separate Lagrange interpolation polynomials of degree $N$ for each subinterval, interpolating the solution at $N+1$ grid points surrounding the subinterval. The numerical results of both methods are summarized in Table \ref{tab5} for different values of $m$. As observed, our proposed method achieves higher accuracy compared with both methods presented in \cite{Ferras2020} and \cite{Lillemae2025}.
\begin{table}[h]
\renewcommand{\arraystretch}{1.3}
  \centering
  \caption{Computed errors for Example \ref{exm5} with $m$ subintervals: hybrid collocation \cite{Ferras2020} ($N=4$, $\theta = \tfrac{1}{4}, \tfrac{1}{2}, \tfrac{2}{3}$) and central part interpolation \cite{Lillemae2025} ($N=4$, $\theta = \tfrac{2}{5}$).}\label{tab5}
  \begin{tabular}{|c||c||c|c||c|c||c|c|}
  \hline
  \multicolumn{6}{|c||}{Hybrid collocation \cite{Ferras2020}} & \multicolumn{2}{|c|}{Central part interpolation \cite{Lillemae2025}} \\
  \hline
  \multicolumn{2}{|c||}{$\theta=\tfrac{1}{4}$} & \multicolumn{2}{|c||}{$\theta=\tfrac{1}{2}$} & \multicolumn{2}{|c||}{$\theta=\tfrac{2}{3}$} & \multicolumn{2}{|c|}{$\theta=\tfrac{2}{5}$} \\
  \hline
   $m$ & $\text{max error}$ & $m$ & $\text{max error}$ & $m$ & $\text{max error}$ & $m$ & $\text{max error}$ \\
  \hline
15 & $2.11\times 10^{-6}$ & 26 & $1.94\times 10^{-7}$ & 32 & $5.53\times 10^{-8}$ & 64 & $6.07\times 10^{-6}$ \\
34 & $1.45\times 10^{-7}$ & 59 & $1.27\times 10^{-8}$ & 71 & $3.58\times 10^{-9}$ & 128 & $2.24\times 10^{-7}$ \\
75 & $9.43\times 10^{-9}$ & 128 & $8.13\times 10^{-10}$ & 155 & $2.28\times 10^{-10}$ & 256 & $7.61\times 10^{-9}$ \\
166 & $6.00\times 10^{-10}$ & 278 & $5.14\times 10^{-11}$ & 331 & $1.44\times 10^{-11}$ & 512 & $2.47\times 10^{-10}$ \\
  \hline
  \end{tabular}
\end{table}
\\

Brugnano et al. \cite{BrugnanoGurioliIavernaro2025,BrugnanoGurioliIavernaroVikerpuur2025,BrugnanoGurioliIavernaroVikerpuur2025FDE} made a significant contribution to the study of initial value FDEs by developing a class of Runge–Kutta type methods, called fractional Hamiltonian boundary value methods (FHBVMs), along with providing the corresponding Matlab codes available online\footnote{\url{https://people.dimai.unifi.it/brugnano/fhbvm/}}. In the following two examples, we make a comparison between the M\"untz-Jacobi Galerkin method and the approaches presented in \cite{BrugnanoGurioliIavernaroVikerpuur2025FDE}.
\begin{example}\label{exm6}
     \cite[Problem 1]{BrugnanoGurioliIavernaroVikerpuur2025FDE} Given the following linear FDE
\begin{equation*}
\begin{cases}
D^{\theta}_Cv_{1}(t)=-v_{1}(t)+1,\\
v_{1}(0)=10, \quad \Lambda=[0,10^{3}],
\end{cases}
\end{equation*}
with the exact solution
\begin{align*}
v_1(t)=9 E_{\frac{1}{2}}(-t^{\frac{1}{2}})+1,
\end{align*}
where $E_{c}(\cdot)$ denotes the one-parameter Mittag-Leffler function, we employ our presented method to solve this equation, and the numerical results are shown in Figure \ref{fg151} and Table \ref{tab_cpu}.
\end{example}
\begin{figure}[!ht]
\centering
\begin{minipage}{0.48\textwidth}
    \centering
    \includegraphics[width=\linewidth]{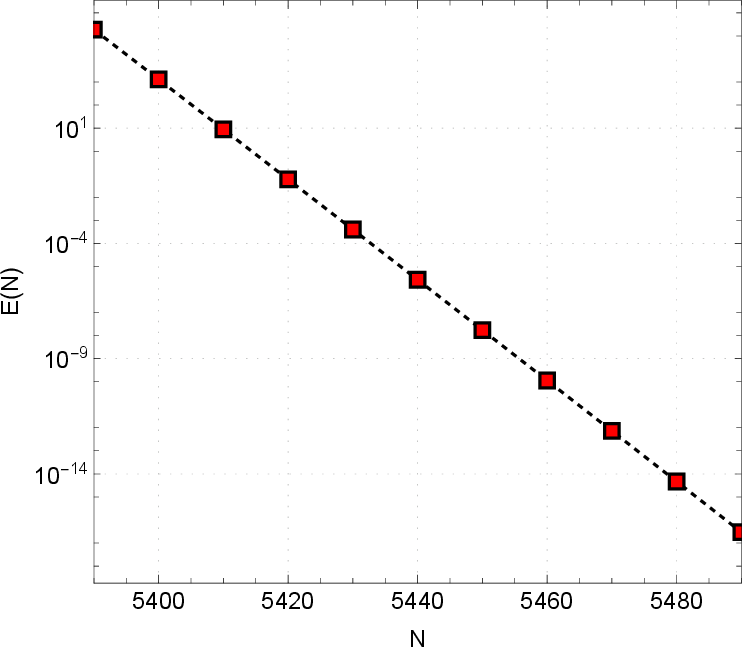}
    \caption{Semi-logarithmic plot of the numerical errors versus $N$ for Example \ref{exm6}.}
    \label{fg151}
\end{minipage}\hfill
\begin{minipage}{0.48\textwidth}
    \centering
    
    \vspace{3.75cm} 
      \renewcommand{\arraystretch}{1.3}
    \begin{tabular}{c|c}
        $N$ & CPU Time (s) \\
        \hline
        685 & 1.41 \\
        1370 & 6.06 \\
        2740 & $2.67\times 10^{+1}$ \\
        5480 & $9.18\times 10^{+1}$\\
    \end{tabular}
    \captionof{table}{Elapsed CPU time versus $N$ for Example \ref{exm6}.}
    \label{tab_cpu}
\end{minipage}
\end{figure}
To assess the performance of the proposed method, we compare it with the approaches introduced by Brugnano et al. \cite{BrugnanoGurioliIavernaroVikerpuur2025FDE}. In this work, each approach is addressed by the name of the corresponding Matlab code. A summary of the numerical results obtained with these methods is reported in Table \ref{tab:accuracy_time}. For further details, in particular the plot of the execution time versus accuracy, we refer the reader to \cite[Figure 11]{BrugnanoGurioliIavernaroVikerpuur2025FDE}.
\begin{table}[!ht]
\renewcommand{\arraystretch}{1.3}
\centering
\begin{tabular}{|c||c||c|}
\hline
Method (Code) & Digits of Accuracy & CPU Time (s) \\
\hline
Predictor–corrector (\texttt{fde12}) \cite{Garrappa2010} & 6 & $5.0\times 10^{+2}$ \\
Fractional linear multi-step (\texttt{flmm2}) \cite{Garrappa2015}  & 10 & $7.0\times 10^{+1}$ \\
Spline collocation (\texttt{fcoll}) \cite{Cardone2021} & 12 & 0.50 \\
Two-step Spline collocation (\texttt{tsfcoll}) \cite{Cardone2021b} & 14 & 6\\
FHBVM (\texttt{fhbvm}) \cite{BrugnanoGurioliIavernaro2025,BrugnanoGurioliIavernaroVikerpuur2025} & 14 & $5.0\times 10^{-2}$ \\
\hline
\end{tabular}
\caption{Numerical results obtained through the different methods presented in \cite{BrugnanoGurioliIavernaroVikerpuur2025FDE} for Example \ref{exm6}.}
\label{tab:accuracy_time}
\end{table}

Using the M\"untz-Jacobi Galerkin method, we achieve 15 digits of accuracy in less than $8.0 \times 10^{+1}$ (s) of CPU time, which is the highest among all methods reported in Table~\ref{tab:accuracy_time}. Although the execution time is shorter than that of the predictor–corrector method (\texttt{fde12}), it is longer than the other approaches, while delivering the most accurate results.

It should be noted that the proposed method remains stable and convergent even for problems defined on very long domains, as in this example, or for problems with very highly oscillatory solutions (see Example \ref{exm7}). However, achieving small errors in these cases requires very large approximation degrees. For instance, in this example, an error of order $10^{-16}$ is reached only for $N=5480$. This inevitably leads to increased computational cost and longer CPU times, indicating that such problems pose greater numerical challenges for the method.
\begin{example}\label{exm7}
    \cite[Problem 4]{BrugnanoGurioliIavernaroVikerpuur2025FDE} As the last problem we consider a stiffly oscillatory problem
\begin{align*}
\begin{cases}
D^{\frac{1}{2}}_{C} \bold{V}(t)= A \bold{V}(t), \\
\bold{V}(0)= \bold{V}_{0}, \quad  \Lambda = [0,T],
\end{cases}
\end{align*}
where
\begin{align*}
A= \frac{1}{8}
\begin{bmatrix}
41  & 41&  -38 & 40 &  -2 \\
-79  &   81 & 2 & 0 &   -2 \\
20 &  -60 &   20 & -20 & -8 \\
  -22 & 58 &  -24 &  20 &  -4 \\
 1 &  1 &  -2 &  -4 &    -2
\end{bmatrix}, \quad
\bold{V}(t)=\big[v_{j}(t)\big]_{j=1}^{5}, \quad
\bold{V}_{0}=
\begin{bmatrix}
1 \\ 2 \\ 3 \\ 4 \\ 5
\end{bmatrix}.
\end{align*}
The exact solution is given by $\bold{V}(t)=E_{\frac{1}{2}}(At^{\frac{1}{2}})\bold{V}_{0}$.
\end{example}
The problem is solved via the suggested method for $T=2$, and the corresponding results are depicted in Figures \ref{fg7600}, \ref{fg7500}, and Table \ref{cpu5}.
\begin{figure}[!ht]
\centering
\begin{minipage}{0.48\textwidth}
    \centering
    \includegraphics[width=\linewidth]{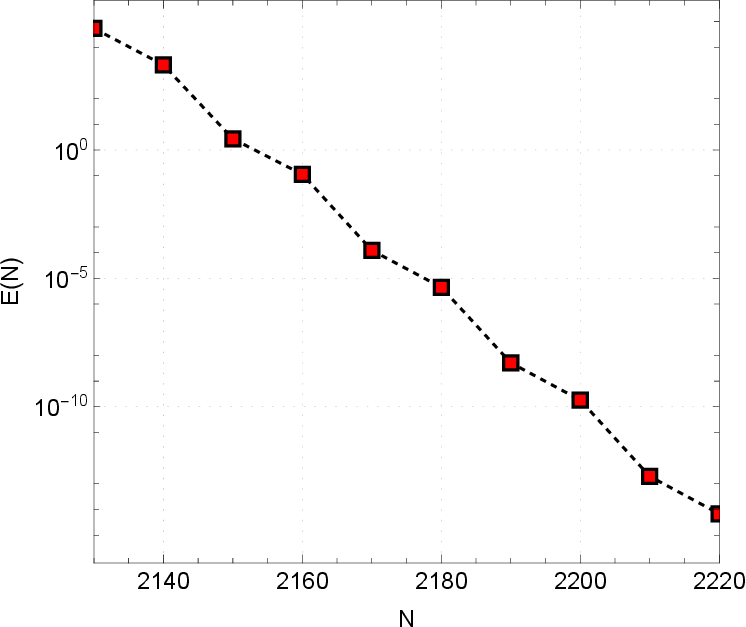}
    \caption{Semi-logarithmic plot of the numerical errors versus $N$ for Example \ref{exm7}.}
    \label{fg7600}
\end{minipage}\hfill
\begin{minipage}{0.48\textwidth}
    \centering
    
    \vspace{3.5cm} 
      \renewcommand{\arraystretch}{1.3}
    \begin{tabular}{c|c}
        $N$ & CPU Time (s) \\
        \hline
        280 & 3.77 \\
        560 & $1.70\times 10^{+1}$ \\
        1120 & $6.84\times 10^{+1}$ \\
        2240 & $2.80\times 10^{+2}$\\
    \end{tabular}
    \captionof{table}{Elapsed CPU time versus $N$ for Example \ref{exm7}.}
    \label{cpu5}
\end{minipage}
\end{figure}
\begin{figure}[!ht]
\centerline{\includegraphics[width=8cm]{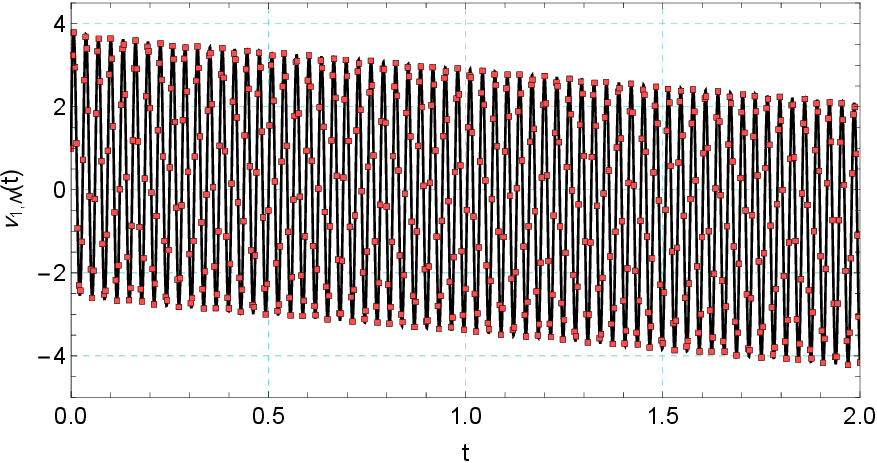},\includegraphics[width=8cm]{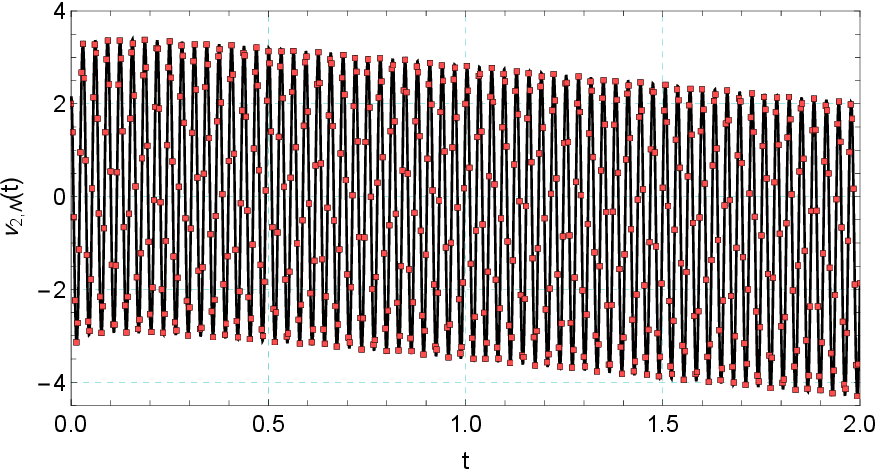}}\centerline{\includegraphics[width=8cm]{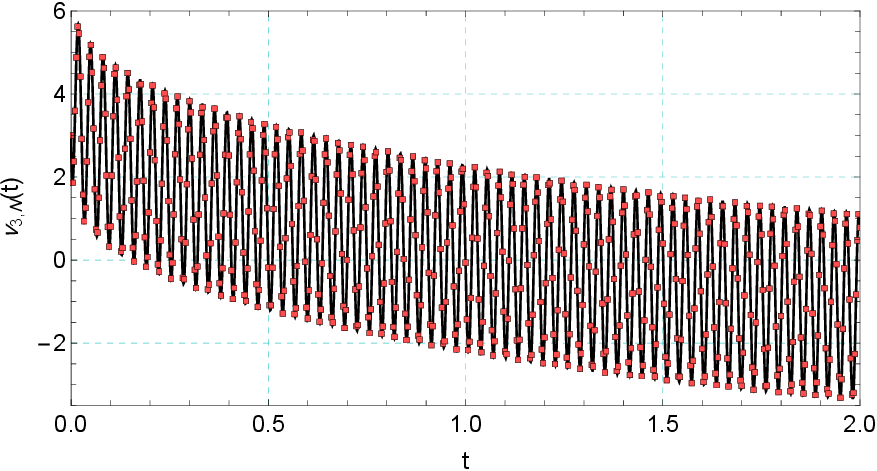},\includegraphics[width=8cm]{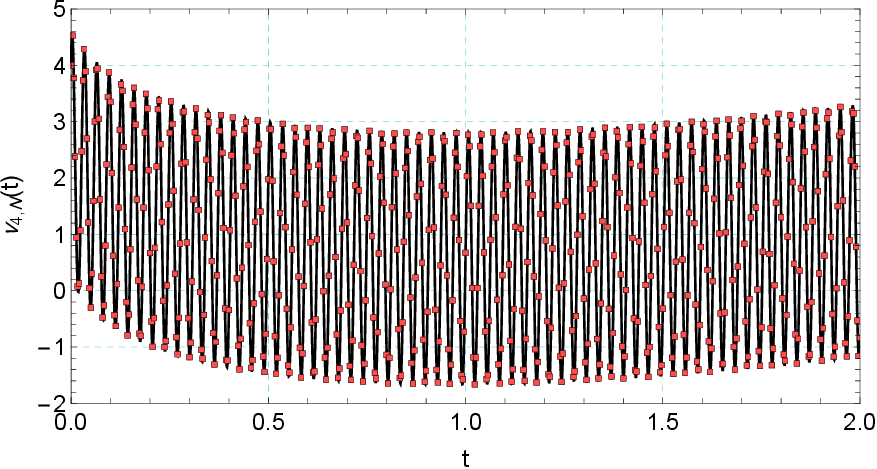}}\centerline{\includegraphics[width=8cm]{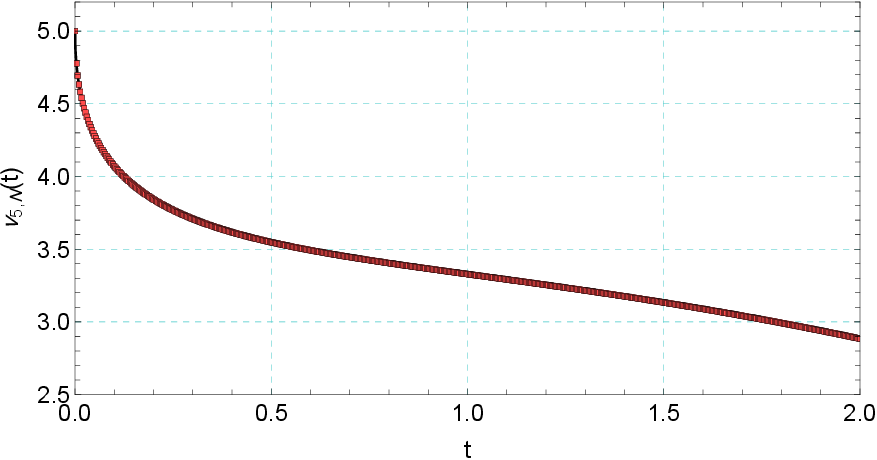}}
\caption{Plots comparing the approximate solutions at $N=2300$ (red squares) and the exact solutions (solid lines) for Example \ref{exm7}.}\label{fg7500}
\end{figure}

Brugnano et al. \cite{BrugnanoGurioliIavernaroVikerpuur2025FDE} addressed this problem using different methods for $T=20$ (see \cite[Figure 14]{BrugnanoGurioliIavernaroVikerpuur2025FDE}), and Table \ref{comp7} summarizes the corresponding results. We were unable to run the example for $T = 20$ via our proposed method on the same computer used for the other experiments, as the longer interval introduces many more oscillations and the problem itself is stiff, exhibiting both fast and slow oscillatory modes. This requires a significantly higher approximation degree to achieve small errors, which in turn, substantially increases the computational cost and necessitates a higher-performance computer. Nevertheless, from the results obtained in both the previous and the current example, it is evident that the FHBVM method is consistently faster than all other methods, while our method achieves higher accuracy. However, the computational cost of our approach becomes considerably larger when solving problems over very long domains with pronounced oscillations. We note that this comparison may not be entirely conclusive, as the other methods were implemented in MATLAB while our method was implemented in Mathematica, so differences in software performance could affect the observed computation times.
\begin{table}[!ht]
\renewcommand{\arraystretch}{1.3}
\centering
\begin{tabular}{|c||c||c|}
\hline
Method (Code) & Digits of Accuracy & CPU Time (s) \\
\hline
Predictor–corrector (\texttt{fde12}) \cite{Garrappa2010} & 3 & $2.50\times 10^{+3}$ \\
Fractional linear multi-step (\texttt{flmm2}) \cite{Garrappa2015}  & 2 & $7.0\times 10^{+2}$ \\
FHBVM (\texttt{fhbvm}) \cite{BrugnanoGurioliIavernaro2025,BrugnanoGurioliIavernaroVikerpuur2025} & 10 & 1.5\\
\hline
\end{tabular}
\caption{Numerical results obtained through the different methods presented in \cite{BrugnanoGurioliIavernaroVikerpuur2025FDE} for Example \ref{exm7}.}
\label{comp7}
\end{table}
\section{Conclusion}
The main purpose of this paper was to design a high-order spectral method to tackle the highly oscillatory and non-smooth solutions of a class of linear SFDEs. We provided a theorem to identify the conditions under which the solution is highly oscillatory and to extract the regularity properties of the exact solutions. Based on this information, we selected the best choice of basis functions with the same asymptotic behavior as the exact solutions. The M\"untz-Jacobi Galerkin method was implemented so that all unknowns were computed through recurrence relations, without the need to solve any potentially ill-conditioned algebraic systems. The choice of basis functions and implementation, together with the provided error analysis, confirmed that the method achieved spectral accuracy under the assumptions of the regularity theorem. 

Although an error analysis was provided in this paper, future research will focus on a thorough investigation of the stability properties of the proposed method.

\section*{Funding}
The research was partially supported by the Research Council KU Leuven (Belgium), 
through project C16/21/002 (Manifactor: Factor Analysis for Maps into Manifolds) and by the Fund for Scientific Research -- Flanders (Belgium), projects G0A9923N (Low rank tensor approximation techniques for up- and downdating of massive online time series clustering) and G0B0123N (Short recurrence relations for rational Krylov and orthogonal rational functions inspired by modified moments).
\section*{Acknowledgements}
The author would like to thank Marc Van Barel and Raf Vandebril for carefully reading a draft of this paper and for their valuable suggestions.
\section*{Declarations}
\textbf{Conflict of interest:} The author declares that there is no conflict of interest nor competing interests.
\bibliographystyle{siam}
\bibliography{references}
\end{document}